\documentclass[10pt,a4paper]{article}
\usepackage[utf8]{inputenc}
\usepackage{amsthm}
\usepackage{amsmath}
\usepackage{amsfonts}
\usepackage{amssymb}
\usepackage{xcolor,mdframed}
\usepackage{graphicx}
\usepackage{tikz}
\usetikzlibrary{matrix}
\usepackage{mathtools}
\usepackage{multirow}
\usepackage{diagbox}
\usepackage{booktabs}
\usepackage{colortbl}%
\usepackage{makecell}
\usepackage[column=0]{cellspace} 
\usepackage[ruled,linesnumbered]{algorithm2e}
\newtheorem{lemma}{Lemma}
\newtheorem{theorem}{Theorem}

\newtheorem{definition}{Definition}
\newlength\figureheight
\newlength\figurewidth
\usepackage{float}
\usepackage[english]{babel}
\usepackage{csquotes}
\usepackage{authblk}
\usepackage{pgfplots}
\usepackage{lscape} 
\usepackage{cite}
\colorlet{tableheadcolor}{gray!25} 
\colorlet{tablerowcolor}{gray!10} 
\makeatletter
\renewcommand*\env@matrix[1][*\c@MaxMatrixCols c]{%
	\hskip -\arraycolsep
	\let\@ifnextchar\new@ifnextchar
	\array{#1}}
\makeatother

\makeatletter
\def\keywords{\xdef\@thefnmark{}}
\makeatother

\makeatletter
\def\@mathmeasure#1#2#3{\setbox#1\hbox{%
		\m@th$#2#3$}}
\makeatother
\usepackage[normalem]{ulem} 
\usepackage{cite}

\title{A Lanczos-type procedure for tensors}
\author{S. Cipolla\thanks{School of Mathematics, The University of Edinburgh, 
		Peter Guthrie Tait Road, EH9 3FD,  Edinburgh, UK 
		({\tt scipolla@ed.ac.uk})}, S. Pozza\thanks{Faculty of Mathematics and Physics, Charles University, Sokolovsk\'a 83, 186 75 Praha 8, Czech Republic ({\tt pozza@karlin.mff.cuni.cz})}, M. Redivo-Zaglia\thanks{ Department of Mathematics ``Tullio Levi-Civita", University of Padova, Via Trieste 63,
		35121, Padova - Italy ({\tt michela@math.unipd.it})    }, N. Van Buggenhout\thanks{Faculty of Mathematics and Physics, Charles University, Sokolovsk\'a 83, 186 75 Praha 8 ({\tt buggenhout@karlin.mff.cuni.cz})}}
\begin{document}
	\maketitle
	\begin{abstract}
		The solution of linear non-autonomous ordinary differential equation systems (also known as the time-ordered exponential) is a computationally challenging problem arising in a variety of applications. In this work, we present and study a new framework for the computation of bilinear forms involving the time-ordered exponential. Such a framework is based on an extension of the non-Hermitian Lanczos algorithm to 4-mode tensors. Detailed results concerning its theoretical properties are presented. Moreover, computational results performed on real world problems confirm the effectiveness of our approach.
	\end{abstract}

{ \footnotesize
	\noindent \keywords{\textbf{Keywords}: non-Hermitian Lanczos algorithm, $\star$-Lanczos algorithm, Lanczos-type procedures for tensors, time-ordered exponential} \\
}
	
	\section{Introduction}\label{sec:intro}
	In this paper, we present an extension of the non-Hermitian Lanczos algorithm (see, e.g., \cite{GolMeuBook10}) where the inputs are $4$-mode tensors $\boldsymbol{\mathcal{A}} \in \mathbb{C}^{N \times N \times M \times M}$ and vectors $\mathbf{w}, \mathbf{v} \in \mathbb{C}^{N}$ so that $\mathbf{w}^H \mathbf{v} \neq 0$.
	We aim to use the introduced algorithm to approximate the
		bilinear form $\mathbf{w}^H \boldsymbol{\mathsf{U}}(t) \mathbf{v}$, where $\boldsymbol{\mathsf{U}}(t) \in \mathbb{C}^{N \times N}$ is the so-called 
		time-ordered exponential, i.e., the solution of the ordinary differential equation
	\begin{equation}\label{eq:ODE}
		\frac{d}{dt}\boldsymbol{\mathsf{U}}(t) = \boldsymbol{\mathsf{A}}(t) \boldsymbol{\mathsf{U}}(t), \quad \boldsymbol{\mathsf{U}}({a})=I_N, \quad t \in I = [{a},b],
	\end{equation}
	where {$I_N\in\mathbb{R}^{N\times N}$ is the identity matrix and} $\boldsymbol{\mathsf{A}}(t) \in \mathbb{C}^{N \times N}$ is a smooth matrix-valued function defined on the real interval $I$.
     Equation \eqref{eq:ODE} can emerge in a variety of {applications}. For example, its solution is crucial in quantum physics where the matrix $\boldsymbol{\mathsf{A}}(t)$ corresponds to the Hamiltonian operator. Situations where $\boldsymbol{\mathsf{U}}(t)$ has no accessible expression are frequent in literature, see, e.g., \cite{Blanes2009,Autler1955,Shirley1965,Lauder1986}. 
	For instance, in Nuclear Magnetic Resonance (NMR) experiments, the associated bilinear form $\mathbf{w}^H \boldsymbol{\mathsf{U}}(t) \mathbf{v}$ represents the measurement of changes in an applied magnetic field caused by nuclear spins that are excited with electromagnetic waves, i.e, spectroscopy \cite{Le08,HaSp98}.
	Other applications are found in control theory, filter design, and model reduction problems \cite{Reid63,kwaSiv72,Corless2003,Blanes15,BenEtAll17}.
	In the mentioned applications, the matrix $\boldsymbol{\mathsf{A}}(t)$ is often large-to-huge and sparse. 
	The introduced algorithm is motivated and theoretically supported by a new expression for the bilinear form. This expression is given by combining the two symbolic methods known as Path-sum and $\star$-Lanczos algorithm \cite{Giscard2015, GisPozInv19,GisPoz21,GisPoz22}.
	Given the matrix-valued function $\boldsymbol{\mathsf{A}}(t)$ and the vectors $\mathbf{w}, \mathbf{v}$, the two symbolic methods produce an expression for the bilinear form $\mathbf{w}^H \boldsymbol{\mathsf{U}}(t) \mathbf{v}$ composed of a finite and treatable number of integrals and scalar integral equations. 
	To our knowledge, no other symbolic method can express the bilinear form with a treatable finite number of integral subproblems. Two commonly used alternative expressions are given by the Magnus series, i.e., an infinite series of nested integrals (e.g., \cite{Magnus1954}), {and} by the Floquet theory, where the solution of an infinite system of coupled linear differential equation is required (e.g., \cite{Blanes2009}).
	
	The integrals and the integral equations generated by the $\star$-Lanczos and Path-sum methods do not always have an easily accessible solution. As a consequence, a numerical approach is needed. A possible strategy for the numerical approximation of the mentioned integrals and the integral equations is outlined in \cite{GisPoz22} and it is based on the discretization of the {interval} $I$ into $M-1$ equispaced subintervals. The algebraic objects resulting from the discretization strategy are a $4$-mode tensor $\boldsymbol{\mathcal{A}}$ (corresponding to $\boldsymbol{\mathbf{A}}(t)$) and  the $3$-mode tensors $V, W$  (corresponding to $\mathbf{v}, \mathbf{w}$).
	
	The outputs obtained by combining the $\star$-Lanczos algorithm with the mentioned discretization strategy are mathematically equivalent to the outputs of the tensor Lanczos algorithm presented here with, as inputs, $\boldsymbol{\mathcal{A}}, \mathbf{v}, \mathbf{w}$.
	The main goal of {this paper} is to show that, in fact, the tensor Lanczos algorithm can converge to the outcome of the $\star$-Lanczos method within an {accuracy of the same order as the discretization strategy}. {Moreover}, the reported numerical experiments will show that the approximation of $\mathbf{w}^H \boldsymbol{\mathsf{U}}(t) \mathbf{v}$ obtained by combining the tensor Lanczos with the discretized Path-sum approach also converges to the solution within {the order of the discretization}.
	Naturally, many numerical methods for the solution of non-autonomous ODEs can be found in literature, see, for instance, \cite{HocLub99,BudAl99,IseAl2000,Ise02,Ise04,DegSch06,Cohetal2006,Blanes2009,BadEtAl16,Bla17}.  For large matrices, these numerical methods are known to be highly demanding both in terms of computational cost and storage. This motivates the research of novel approaches suitable for large-scale problems.
	In order to be competitive with the most advanced techinques, tensor Lanczos needs to be used in combination with more accurate discretization schemes. Development of suitable, faster converging discretization schemes is an ongoing research and out of the scope of this work. At the same time, it is important to note that the algorithm here proposed is part of a wider class of tensor extensions of Krylov subspace methods that recently appeared in the literature, see e.g., \cite{Huang2019,Guide2021,Feng2020,Reichel2021}.

	The Lanczos-type process we introduce can also be equivalently written as a block Lanczos method since the $4$-mode tensor $\boldsymbol{\mathcal{A}}$ can be seen as a block matrix; information about block Krylov subspace methods can be found, e.g., in \cite{etnavol47pp100126}. Despite this fact, we prefer to interpret such a block structure in a \textit{tensorial}  fashion. Indeed, the tensorial approach has a direct translation in terms of a discretized $\star$-Lanczos algorithm. Moreover, as we will experimentally show,
	this interpretation is motivated by observing that several tensors from real world examples related with \eqref{eq:ODE} are characterized by a low parametric approximation known as  \emph{Tensor Train} decomposition (TT) \cite{MR2837533,MR2566459}. 
	Such a low-parametric approximation allows to efficiently manipulate and store the tensors. This paves the path for {further} improvements of our proposal, where the TT structure is fully exploited in the Lanczos-type procedure.
	Examples of tensor Krylov subspace methods combined with the TT decomposition can be found in \cite{Feng2020,Ruymbeek2022},
	further motivating  our \textit{tensor-based} point of view.
	
	More in detail, this work is organized as follows.
	Preliminaries and definitions of tensor operations are given in Section \ref{sec:pre}. In Section \ref{sec:lanczos} we discuss how to construct the non-Hermitian Lanczos {procedure} for tensors and we prove several crucial properties. In Section \ref{sec:break} we  discuss the breakdown issue which typically arises when working with non-Hermitian Lanczos approaches. Numerical experiments are presented in Section \ref{sec:ex} where we also give several examples exposing the low-rank TT structure of {the considered} tensors $\boldsymbol{\mathcal{A}}$. Section \ref{sec:conclusions} concludes the paper and Appendix \ref{appendix} contains several proofs.

	\section{Preliminaries}\label{sec:pre}
	In this work, we use a notation borrowed from  Matlab\textsuperscript{\textregistered}.
		Fixing  $i_1\in \{1,\dots,N_1\}$ and $i_2\in \{1,\dots,N_2\}$, if $\boldsymbol{\mathcal{A}} \in \mathbb{C}^{N_1 \times N_2 \times M \times M }$, then $\boldsymbol{\mathcal{A}}_{i_1,i_2,:,:}$ stands for the matrix 
		\begin{equation*}
			\boldsymbol{\mathcal{A}}_{i_1,i_2,:,:}:=\left[\boldsymbol{\mathcal{A}}_{i_1,i_2,j_1,j_2}\right]_{j_1,j_2= 1}^{M}.
		\end{equation*}
		This notation similarly applies to $3$-mode tensors, matrices, and vectors.
		Table~\ref{tab:notation_summary} summarizes the notation used in the paper.
	\begin{table}[ht!]
		\centering
		\begin{tabular}{ll}
			\hline
			Symbol &  Description  \\
			\hline
			$\boldsymbol{\mathsf{A}}(t), \boldsymbol{\mathsf{U}}(t)$ & Matrix-valued functions \\
			$\boldsymbol{\mathcal{A}},\boldsymbol{\mathcal{B}}...$ &        $4$-mode tensors   \\
			${A},{B}...$ &                     $3$-mode tensors \\
			$\boldsymbol{\alpha},\boldsymbol{\beta}...$  & $M \times M$ Matrices\\
			$I_M$ & $M \times M$ identity matrix \\
			$\mathbf{a},\mathbf{b}$&    Vectors \\
			$\boldsymbol{\mathcal{A}}^{k_*}$ & $k$-th $*$-power of $\boldsymbol{\mathcal{A}}$\\ 
			$\boldsymbol{\mathcal{I}}_{*}$ & $*$-identity\\
			\hline
		\end{tabular}
		\caption{Summary of notation.} \label{tab:notation_summary}
	\end{table}
		
	In the following, we define several tensorial operations, which can be seen as generalizations of the usual products involving matrices and vectors. We summarize them in Table \ref{table:prod}.
	\begin{table}[ht!]\label{table:prod}
		\begin{tabular}{clccl}
			\hline
			Symbol & Name &  $n, m$ & $k$ & Generalizing  \\
			\hline
			$\ast$ & $\ast$-Tensor product & $4, 4$ & $4$ & Matrix-matrix product \\
			$\ast$ & Tensor-Hypervector product & $4, 3$ & $3$ & Matrix-vector product \\
			$\ast$ & Hypervector inner-product & $3, 3$ & $2$ & Inner-product \\
			$\times$ & Tensor-matrix product & $4, 2$ & $4$ & Matrix-scalar product \\
			$\times$ & Hypervector-matrix product & $3, 2$ & $3$ & Vector-scalar product \\
			$\otimes$ & Vector-to-Hypervector product & $1, 2$ & $3$ & Kronecker product    \\
			\hline
		\end{tabular}
		\caption{Each of the considered \textit{products} involves two tensors with $n$ and $m$ modes {and it gives as an outcome a} tensor with $k$-modes.}
		\label{table:ops}
	\end{table}	
	In the following definitions we consider the tensors $\boldsymbol{\mathcal{A}} \in \mathbb{C}^{N_1\times N_2 \times M \times M },\boldsymbol{\mathcal{B}} \in \mathbb{C}^{N_2 \times N_3 \times M \times M }$,  $A \in \mathbb{C}^{N_2 \times M \times M }$, $B \in \mathbb{C}^{N_1 \times M \times M }$, $\boldsymbol{\alpha} \in \mathbb{C}^{M \times M}$. Moreover, the indices $i_1\in \{1,\dots,N_1\}$  $i_2\in \{1,\dots,N_2\}$ are fixed.
	\begin{definition}[$*$-Tensor product]
		The product $(\boldsymbol{\mathcal{A}}*\boldsymbol{\mathcal{B}})\in \mathbb{C}^{N_1 \times N_3 \times M \times M } $ is defined as
		\begin{equation*}
			(\boldsymbol{\mathcal{A}}*\boldsymbol{\mathcal{B}})_{i_1,i_2,:,:}:= \sum_{k=1}^{N_2}{\boldsymbol{\mathcal{A}}_{i_1,k,:,:}}\;\boldsymbol{\mathcal{B}}_{k,i_2,:,:} .
		\end{equation*}
	\end{definition}
	
	\begin{definition}[Tensor-Hypervector product] \label{def:THV}
		The product $(\boldsymbol{\mathcal{A}}*A) \in \mathbb{C}^{N_1 \times M \times M }$  is defined as
		\begin{equation*}
			(\boldsymbol{\mathcal{A}}*A)_{i_1,:,:}:= \sum_{k=1}^{N_2}{\boldsymbol{\mathcal{A}}_{i_1,k,:,:}}\;A_{k,:,:}.
		\end{equation*}
		We also need to define the action of a $3$-mode tensor from the left. Every tensor with three modes that acts, will act, or is the outcome of a $*$-product from the left, will be denoted with a "$D$" (dual) as apex {, and, in the remainder of this work, we will use $B^D_{k,:,:}$ to denote $(B^D)_{k,:,:}$ }. We define, $\; (B^D*\boldsymbol{\mathcal{A}})^D\in \mathbb{C}^{ N_2 \times M \times M }$ as
		\begin{equation*}
			(B^D*\boldsymbol{\mathcal{A}})^D_{i_2,:,:}:= \sum_{k=1}^{N_1}{{B^D_{k,:,:}}\;\boldsymbol{\mathcal{A}}_{k,i_2,:,: }}.
		\end{equation*}
	\end{definition}

	Note that the following $4$-mode tensor is the identity for $*$-products introduced above
	\begin{equation*}
		\mathbb{C}^{N_1 \times N_1 \times M \times M} \ni	(\boldsymbol{\mathcal{I}}_{*})_{i_1,i_2,: , :}:= \left \{\begin{array}{ll}
			I_M, & \hbox{ if } i_1=i_2 \\
			0_M, & \hbox{otherwise}
		\end{array}   \right. . 
	\end{equation*}
	\begin{definition}[Hypervector inner-product]
		The product $(B^D*A) \in \mathbb{C}^{M\times M}$ is defined as
		\begin{equation*}
			(B^D*A)_{:,:}:= \sum_{k=1}^{N_1}{B^D_{k,:,:}\;A_{k,:,: }}.
		\end{equation*}
	\end{definition}
	
	\begin{definition}[Tensor-matrix product]
		The products $(\boldsymbol{\mathcal{A}} \times \boldsymbol{\alpha}),\; ( \boldsymbol{\alpha} \times \boldsymbol{\mathcal{A}} ) \in  \mathbb{C}^{N_1 \times N_2 \times M \times M }$ are defined as
		\begin{equation*}
			(\boldsymbol{\mathcal{A}} \times \boldsymbol{\alpha})_{i_1,i_2,:,:}:= {\boldsymbol{\mathcal{A}}_{i_1,i_2,:,: }}\;\boldsymbol{\alpha} \quad \hbox{ and } \quad ( \boldsymbol{\alpha} \times \boldsymbol{\mathcal{A}} )_{i_1,i_2,:,:}:= {\boldsymbol{\alpha} \,\boldsymbol{\mathcal{A}}_{i_1,i_2,:,: }}.
		\end{equation*}
	\end{definition}
	
	\begin{definition}[Hypervector-matrix product]
		The products 	$({A} \times \boldsymbol{\alpha}), \; ( \boldsymbol{\alpha} \times {A} ) \in  \mathbb{C}^{N_2 \times M \times M }$ are defined as
		\begin{equation*}
			({A} \times \boldsymbol{\alpha})_{i_1,:,:}:= {{A}_{i_1,:,: }}\boldsymbol{\alpha} \quad \hbox{ and } \quad ( \boldsymbol{\alpha} \times {A} )_{i_1,:,:}:= {\boldsymbol{\alpha}{A}_{i_1,:,: }}.	\end{equation*}
	\end{definition}
	
	\begin{definition}[Vector-to-Hypervector]
		Given $\mathbf{a} \in \mathbb{C}^N$ we define the product $ A=\mathbf{a}\otimes I_M \in \mathbb{C}^{N \times M \times M}$ as
		$$A_{i_1,:,:}=\mathbf{a}_{i_1}I_M \quad  i_1 \in \{1,\dots,N\}.$$
	\end{definition}
	Note that rearranging $A$ as a block matrix, we get the usual Kronecker product.
	All the products are clearly distributive with respect to the usual addition. On the other {hand}, the associativity of some of the products is less obvious. Therefore, we state it in the following Lemma \ref{lemma:properties_defintions1}, postponing its proof to Appendix {\ref{app:assoc}}.
	\begin{lemma}\label{lemma:properties_defintions1}
			The following statements show that the tensor-tensor and tensor-hypervector $\ast$-products are associative. 
			\begin{itemize}
				\item Given $\boldsymbol{\mathcal{A}} \in \mathbb{C}^{N_1 \times N_1 \times M \times M}$, $A \in \mathbb{C}^{ N_1 \times M \times M}$ we have
				\begin{equation*}
					(\boldsymbol{\mathcal{A}}*\boldsymbol{\mathcal{A}})*{A}= \boldsymbol{\mathcal{A}}*(\boldsymbol{\mathcal{A}}*{A}){.}
				\end{equation*}
				\item Given  $B^D \in \mathbb{C}^{N_1 \times  M \times M}$, ${A} \in \mathbb{C}^{N_2 \times M \times M }$, $\boldsymbol{\mathcal{A}} \in \mathbb{C}^{N_1\times N_2 \times M \times M }$, then
				\begin{equation*}
					(B^D * \boldsymbol{\mathcal{A}})^D *A=B^D * (\boldsymbol{\mathcal{A}} *A).
				\end{equation*}
				\item Given  $\boldsymbol{\mathcal{A}} \in \mathbb{C}^{N_1\times N_2 \times M \times M },\;\boldsymbol{\mathcal{B}} \in \mathbb{C}^{N_2 \times N_3 \times M \times M },\; \boldsymbol{\mathcal{C}} \in \mathbb{C}^{N_3\times N_1 \times M \times M }$, then
				\begin{equation*}
					(\boldsymbol{\mathcal{C}}*\boldsymbol{\mathcal{A}})*\boldsymbol{\mathcal{B}}=\boldsymbol{\mathcal{C}}*(\boldsymbol{\mathcal{A}}*\boldsymbol{\mathcal{B}}).
				\end{equation*}
			\end{itemize}
		\end{lemma}
	Having introduced the required products and their basic properties, we are ready to derive the tensor non-Hermitian Lanczos algorithm.
	
	\section{The Lanczos-Type Process}\label{sec:lanczos}
	Using the operations given in Table \ref{table:ops}, {we propose} a sensible generalization of Krylov subspaces where, instead of the usual matrix-vector product, the tensor-hypervector product is used to generate the subspaces.
	Section \ref{sec:KrylovSubspaces} describes these tensor Krylov-type subspaces in detail and defines biorthogonal bases for them.
	A Lanczos-type algorithm which generates these biorthogonal bases is proposed in Section \ref{sec:tensorLanczos}.
	In Section \ref{sec:tensorLanczosProperties} two important properties of the classical Lanczos algorithm are generalized, namely, the tensor representation of the three-term recurrence relations for the biorthogonal bases and the matching moment property.
	The computational cost and storage requirements of the algorithm are discussed in Section~\ref{sec:num_prop}.
	
	\subsection{Krylov-type tensor subspaces}\label{sec:KrylovSubspaces}
	
	Consider the tensor $\boldsymbol{\mathcal{A}}\in \mathbb{C}^{N \times N \times M \times M}$. We define the \textit{polynomials} of degree {$\ell$} of $\boldsymbol{\mathcal{A}}$ as
	\begin{align*}
		p(\boldsymbol{\mathcal{A}}) &:= \sum_{k=0}^{{\ell}}   \boldsymbol{\mathcal{A}}^{k_*} \times \boldsymbol{\alpha}_k ,\\
		p^{D}(\boldsymbol{\mathcal{A}}) &:=   \sum_{k=0}^{{\ell}} \boldsymbol{\alpha}_k^H \times \boldsymbol{\mathcal{A}}^{k_*}   ,
	\end{align*}
	where $\boldsymbol{\mathcal{A}}^{k_*}$ stands for $k$ $\ast$-multiplications of $\boldsymbol{\mathcal{A}}$ by itself, and $\boldsymbol{\alpha}_k^H$ is the conjugate transpose of $\boldsymbol{\alpha}$.
	Given the tensors ${A}\in \mathbb{C}^{N \times M \times M}$,  $B\in \mathbb{C}^{N \times M \times M}$
	we can define the Krylov-type subspaces 
	\begin{align*}
		\mathcal{K}_n(\boldsymbol{\mathcal{A}},A)&:= \{p(\boldsymbol{\mathcal{A}})*A \hbox{ s.t. }  deg(p) \leq n-1 \},\\
		\mathcal{K}_n^{D}(B^D,\boldsymbol{\mathcal{A}})&:= \{B^D * p^{D}(\boldsymbol{\mathcal{A}}) \hbox{ s.t. }  deg(p^D) \leq n-1 \}.
	\end{align*}
	Every element in $\mathcal{K}_n(\boldsymbol{\mathcal{A}},A)$ is a tensor in $\mathbb{C}^{N \times M \times M}$ and can be written as 
	\begin{equation*}
		p(\boldsymbol{\mathcal{A}})* A = \sum_{k=0}^{n-1}   (\boldsymbol{\mathcal{A}}^{k_*}\times  \boldsymbol{\alpha}_k ) *A. 
	\end{equation*}
	From now on we will assume {that $A$ is of the form $A=\mathbf{a} \otimes I_M$ for some} $\mathbf{a} \in \mathbb{C}^N$. In this case, the matrices $\boldsymbol{\alpha}_k$ commute with $A$ giving 
	\begin{equation*}
		p(\boldsymbol{\mathcal{A}})* A = \sum_{k=0}^{n-1}   (\boldsymbol{\mathcal{A}}^{k_*}   *A)\times \boldsymbol{\alpha}_k. 
	\end{equation*}
	An analogous result holds for $B^D*p^D(\boldsymbol{\mathcal{A}})$ {when $B^D=\overline{\mathbf{b}} \otimes I_M$ for any $\mathbf{b} \in \mathbb{C}^N$, where $\overline{\mathbf{b}}$ is the conjugated vector.} 
	
	Driven by the analogy with the matrix case, our aim is to build two ``biorthonormal bases'' for the Krylov-type subspaces $\mathcal{K}_{n}(\boldsymbol{\mathcal{A}},A)$ and $ \mathcal{K}_n^{D}(B^D, \boldsymbol{\mathcal{A}})$. 
	The following Definition \ref{def:bases} allows to characterize spaces spanned by 3-mode tensors.
	\begin{definition}\label{def:bases}
		Given $V_1,\dots,V_n \in \mathbb{C}^{N \times M \times M}$, $W_1^D,\dots,W^D_n \in \mathbb{C}^{N \times M \times M}$,  we define the subspaces 
		\begin{align*}
			\langle V_1,\dots, V_n \rangle&:= \left\{V= \sum_{{k}=1}^n V_{{k}} \times \boldsymbol{\eta}_{{k}}, \textrm{ for } \boldsymbol{\eta}_1,\dots, \boldsymbol{\eta}_n \in \mathbb{C}^{M \times M}\right\}; \\
			\langle W_1^D,\dots, W_n^D \rangle&:= \left\{W^D= \sum_{{k}=1}^n  \boldsymbol{\eta}_{{k}} \times W^D_{{k}}, , \textrm{ for } \boldsymbol{\eta}_1,\dots, \boldsymbol{\eta}_n \in \mathbb{C}^{M \times M} \right\}.
		\end{align*}
		We say that $V_1,\dots, V_n$ is a \emph{basis} {for} the subspace $\langle V_1,\dots, V_n \rangle$ and $W_1^D,\dots, W_n^D$ is a \emph{basis} {for} the subspace $\langle W_1^D,\dots, W_n^D \rangle$.
	\end{definition}
	Biorthonormal bases for Krylov-type subspaces are represented by the tensors $\boldsymbol{\mathcal{V}}_{n} \in \mathbb{C}^{N\times n \times M \times M }$ and  $\boldsymbol{\mathcal{W}}_{n}\in \mathbb{C}^{n\times N \times M \times M }$ satisfying
		\begin{equation}\label{eq:biorthcond}
			\boldsymbol{\mathcal{W}}_{n}*\boldsymbol{\mathcal{V}}_{n}=\boldsymbol{\mathcal{I}}_* \in \mathbb{R}^{n \times n \times M \times M},
		\end{equation}
		with the hypervectors  $V_{k}:= (\boldsymbol{\mathcal{V}}_n)_{:,k,:,:}$ and   $W^D_{k}:= (\boldsymbol{\mathcal{W}}_n)_{k,:,:,:}$, for $k=1,\dots,n$, forming, respectively, bases for $\mathcal{K}_{n}(\boldsymbol{\mathcal{A}},A)$ and $ \mathcal{K}_n^{D}(B^D, \boldsymbol{\mathcal{A}})$, i.e., 
		\begin{align*}
			\langle V_1,\dots, V_n \rangle = \mathcal{K}_{n}(\boldsymbol{\mathcal{A}},A), \quad \langle W_1^D,\dots, W_n^D \rangle = \mathcal{K}_n^{D}(B^D, \boldsymbol{\mathcal{A}}).
	\end{align*}
	In the following section we derive such bases by constructing the tensor non-Hermitian Lanczos Algorithm.

	\subsection{The tensor Lanczos process}\label{sec:tensorLanczos}
	Given the inputs $\boldsymbol{\mathcal{A}}\in \mathbb{C}^{N\times N \times M \times M}$ and $\mathbf{v}, \mathbf{w} \; {\in\mathbb{C}^{N}}$, Algorithm \ref{alg:Lanczos} constructs, when no breakdown occurs, the bases $\boldsymbol{\mathcal{W}}_{n}$ and $\boldsymbol{\mathcal{V}}_{n}$,  {for $\mathcal{K}_{n}(\boldsymbol{\mathcal{A}},A)$ and $ \mathcal{K}_n^{D}(B^D, \boldsymbol{\mathcal{A}})$, respectively, which satisfy} the \emph{$\ast$-biorthogonality conditions} \eqref{eq:biorthcond}.
	
	\begin{algorithm}[H]
		\caption{non-Hermitian Lanczos for Tensors} \label{alg:Lanczos}
		\DontPrintSemicolon
		\KwIn{$\boldsymbol{\mathcal{A}}\in \mathbb{C}^{N\times N \times M \times M}$, $\mathbf{v}, \mathbf{w} \in \mathbb{C}^{N}$ such that $\mathbf{w}^H\mathbf{v} = 1$.}
		\KwOut{$ V_1,\dots,V_{n}, \,W_1^D,\dots, W_{n}^D \in \mathbb{C}^{N \times M \times M}$
			spanning respectively $\mathcal{K}_n(\boldsymbol{\mathcal{A}},V)$, $\mathcal{K}_n(W^D, \boldsymbol{\mathcal{A}})$.}
		{Initialize: } $V_0=W_{0}^D=0$, $\boldsymbol{\beta}_1=0$, $V:= \mathbf{v}\otimes I_M,\; {W^D:= \overline{\mathbf{w}}}\otimes I_M,\; \, V_1 := V,\; W_1^D := W^D$ \;
		\For{$k=1, \dots n$}{
			$\boldsymbol{\alpha}_{k} = W_{k}^D * \boldsymbol{\mathcal{A}} * V_{k}$\;
			$\widehat{W}_{k+1}^D = W_{k}^D * \boldsymbol{\mathcal{A}} -  \boldsymbol{\alpha}_{k} \times W_{k}^D - \boldsymbol{\beta}_{k} \times W_{k-1}^D$\;
			${\widehat V}_{k+1} = \boldsymbol{\mathcal{A}} * V_{k} - V_{k} \times \boldsymbol{\alpha}_{k} - V_{k-1} \times \boldsymbol{\gamma}_k$\;
			\text{Set a non singular matrix} $\boldsymbol{\gamma}_{k+1}$\;
			$\boldsymbol{\beta}_{k+1}= (\boldsymbol{\gamma}_{k+1})^{-1}(\widehat{W}_{k+1}^D * \widehat{V}_{k+1})$\;
			\If{$ \boldsymbol{\beta}_{k+1}$ is singular}{\texttt{Stop} \;}
			$V_{k+1}= {\widehat V}_{k+1} \times \boldsymbol{\beta}_{k+1}^{-1}$\;
			$W_{k+1}^D= (\boldsymbol{\gamma}_{k+1})^{-1} \times {\widehat W}_{k+1}^D $\;
		}
	\end{algorithm}
	Details on how the algorithm constructs these bases using three-term recurrences is described below.
	\begin{itemize}
		\item By definition, the first hypervectors of the bases are $W_1^D, V_1$ satisfying $W_1^D*V_1=I_M$;
		\item Consider the vector $\widehat{W}^D_2 \in \mathcal{K}_2(W^D,\boldsymbol{\mathcal{A}})$ given by
		\begin{equation*}
			\widehat{W}^D_2:=W_1^D*\boldsymbol{\mathcal{A}}-\boldsymbol{\alpha}_1 \times W_1^D. 
		\end{equation*}
		Imposing that $\widehat{W}_2^D$ satisfies the $*$-biorthogonal condition $\widehat{W}_2^D*V_1=0$, we have $\boldsymbol{\alpha}_1=W_1^D * \boldsymbol{\mathcal{A}} *V_1$.
		\item Analogously, define the vector $\widehat{V}_2 \in \mathcal{K}_2(\boldsymbol{\mathcal{A}},V)$ given by
		\begin{equation*}
			\widehat{V}_2 := \boldsymbol{\mathcal{A}}*V_1-V_1\times \boldsymbol{\alpha}_1.
		\end{equation*}
		Imposing the $*$-biorthogonality condition, we find the $*$-biorthogonal vectors
		\begin{equation*}
			V_2:=\widehat{V}_2\times \boldsymbol{\beta}_1^{-1} \hbox{ where } \boldsymbol{\beta}_1=\widehat{W}_2^D*\widehat{V}_2=\widehat{W}_2^D*V_1 \; \hbox{ and } W_2=\widehat{W}_2.
		\end{equation*} 
		\item Clearly  $\mathcal{K}_2({\boldsymbol{\mathcal{A}},V})=\langle V_1,V_2 \rangle$ and $\mathcal{K}_2^D(W^D,{\boldsymbol{\mathcal{A}}})=\langle W^D_1,W^D_2 \rangle$.
		\item Now, assume the $*$-biorthonormal bases $V_1,\dots,V_{{k}}$ and $W^D_1,\dots,W^D_{{k}}$ are available. Consider the hypervector
		\begin{equation*}
			\widehat{W}^D_{{k}+1}:={W}_{{k}}^D*\boldsymbol{\mathcal{A}}-\sum_{i=1}^{{k}}\boldsymbol{\boldsymbol{\eta}}_i \times  {W}_{i}^D.
		\end{equation*}
		The matrices $\boldsymbol{\eta}_i$ are determined by the conditions $\widehat{W}_{{k}+1}^D*V_{i}=0$, for $i~=~1,\dots,~k$, which give
		\begin{equation*}
			\boldsymbol{\eta}_i=W_{{k}}^D*\boldsymbol{\mathcal{A}}*V_{i}, \quad \hbox{ for } i=1,\dots,{k}.
		\end{equation*}
		In particular, since $\boldsymbol{\mathcal{A}}*V_i \in \mathcal{K}_{i+1}(\boldsymbol{\mathcal{A}},V )$, we get $\boldsymbol{\eta}_i=0$ for $i=1,\dots,{k}-2$. 
		An analogous argument is valid for $\widehat{V}_{k+1}$.
		This leads to the following three-term recurrences
		
		\begin{subequations} \label{eq:3r}
			\begin{align}
				W_{k+1}^D=W_{k}^D*\boldsymbol{\mathcal{A}}-\boldsymbol{\alpha}_{k}\times W_{k}^D-\boldsymbol{\beta}_{k}\times W_{k-1}^D, \label{eq:3ra}  \\
				V_{k+1}\times \boldsymbol{\beta}_{k+1}=\boldsymbol{\mathcal{A}} *V_{k}- V_{k}\times \boldsymbol{\alpha}_{k}-V_{k-1},	 \label{eq:3rb}
			\end{align}
		\end{subequations}
		with coefficients	
		\begin{equation}\label{eq:3r:coeff}
			\boldsymbol{\alpha}_{k}= W_{k}^D*\boldsymbol{\mathcal{A}}*V_{k}, \;\; \boldsymbol{\beta}_{k+1}=W_{k+1}^D*\boldsymbol{\mathcal{A}}*V_{k}.
		\end{equation}
		\item To prove that $\langle V_1, \dots, V_n \rangle=\mathcal{K}_n(\boldsymbol{\mathcal{A}},V)$ and $\langle W_1^D, \dots, W_n^D \rangle=\mathcal{K}_n(W^D,\boldsymbol{\mathcal{A}})$, it is enough to use induction and the fact that $V_{{k}} \in \mathcal{K}_{{k}}(\boldsymbol{\mathcal{A}},V)$ and $W_{{k}}^D \in \mathcal{K}_{{k}}(W^D,\boldsymbol{\mathcal{A}})$ for all ${k}=1,\dots,n$. 
	\end{itemize}
	Let us finally observe that, should $\boldsymbol{\beta}_{k+1}$ not be invertible, we would get a \textit{breakdown} in the algorithm.
	
	Different rescaling strategies are possible by setting an invertible coefficient $\boldsymbol{\gamma}_{k+1}$ and noticing that
	$$ (\boldsymbol{\gamma}_{k+1})^{-1} \times W_{k+1}^D \ast V_{k+1} \times \boldsymbol{\gamma}_{k+1} = {I_{M}}. $$
	This last observation completes the construction of Algorithm \ref{alg:Lanczos}.

	\subsection{Main Properties of the Tensor Lanczos Algorithm}\label{sec:tensorLanczosProperties}
	It is important to note that the coefficients in the three-term recurrences {\eqref{eq:3r}} can be represented by a sparse 4-mode tensor. To this aim, let us consider $\boldsymbol{\mathcal{T}}_{n} \in \mathbb{C}^{n \times n \times M \times M}$ as the tensor defined as
	\begin{equation}\label{eq:Tcoeff}
		(\boldsymbol{\mathcal{T}}_{n})_{i_1,i_2, :, :}:= \left \{\begin{array}{lll}
			\boldsymbol{\alpha}_{i_1},   & \hbox{ if }  i_1=i_2  & \hbox{ and }  1 \leq i_1 \leq n \\
			\boldsymbol{\gamma}_{i_1}, & \hbox{ if } i_2=i_1+1 & \hbox{ and }  1 \leq i_1 \leq n-1 \\
			\boldsymbol{\beta}_{i_1}, & \hbox{ if } i_2=i_1-1 & \hbox{ and } 2 \leq i_1 \leq n\\
			\boldsymbol{0}, & \hbox{ otherwise } \\
		\end{array}   \right. . 
	\end{equation} 
	where $\boldsymbol{\alpha}_{i_1}, \boldsymbol{\beta}_{i_1}, \boldsymbol{\gamma}_{i_1}$ are the matrices in Algorithm \ref{alg:Lanczos}. The tensor $\boldsymbol{\mathcal{T}}_{n}$ is a generalization of the so-called \emph{(complex) Jacobi matrix} associated with the non-Hermitian Lanczos algorithm; see, e.g., \cite{PozPraStr16} and references therein.
	By using $\boldsymbol{\mathcal{T}}_{n}$, Theorem \ref{thm:3TRR} provides a compact representation of the three-term recurrences constructing the biorthogonal bases.
	\begin{theorem}\label{thm:3TRR}
		The three-recurrences Eqs. \eqref{eq:3r} can be written in the compact form	
		\begin{subequations} \label{eq:recurrences}
			\begin{align}
				\boldsymbol{\mathcal{A}}*\boldsymbol{\mathcal{V}}_n=\boldsymbol{\mathcal{V}}_n*\boldsymbol{\mathcal{T}}_n+ \widetilde{\boldsymbol{\mathcal{V}}}_n  \label{eq:tens3term1}\\
				\boldsymbol{\mathcal{W}}_n*\boldsymbol{\mathcal{A}}=\boldsymbol{\mathcal{T}}_n*\boldsymbol{\mathcal{W}}_n+\widetilde{\boldsymbol{\mathcal{W}}}_n \label{eq:tens3term2}
			\end{align}
		\end{subequations}
		where $  \widetilde{\boldsymbol{\mathcal{V}}}_n  \in \mathbb{C}^{N \times n \times M \times M}$ {is}
		\begin{equation*}
			(\widetilde{\boldsymbol{\mathcal{V}}}_n)_{:,k,:,:}  := \left \{\begin{array}{ll}
				V_{n+1} \times \boldsymbol{\beta}_{n+1},    & \hbox{ if }  k=n  \\
				\boldsymbol{0}, & \hbox{ otherwise } \\
			\end{array}   \right.,
		\end{equation*}
		{and} $\widetilde{\boldsymbol{\mathcal{W}}}_n  \in \mathbb{C}^{n \times N \times M \times M}$ {is}
		\begin{equation*}
			(\widetilde{\boldsymbol{\mathcal{W}}}_n)_{k,:,:,:}  := \left \{\begin{array}{ll}
				\boldsymbol{\gamma}_{n+1} \times W^D_{n+1},    & \hbox{ if }  k=n  \\
				\boldsymbol{0}, & \hbox{ otherwise } \\
			\end{array}   \right. . 
		\end{equation*}
	\end{theorem}
	\begin{proof}
		By direct inspection.	We have, for all $i_1 \in \{1,\dots,N \}$, $i_2 \in \{1,\dots,n-1\}$  
		
		\begin{equation} \label{eq:fK1}
			(\boldsymbol{\mathcal{A}}*\boldsymbol{\mathcal{V}}_n)_{i_1,i_2,:,:}=\sum_{k=1}^{N}\boldsymbol{\mathcal{A}}_{i_1,k,:,:}(\boldsymbol{\mathcal{V}}_n)_{k,i_2,:,:}=\sum_{k=1}^{N}\boldsymbol{\mathcal{A}}_{i_1,k,:,:}({V}_{i_2})_{k,:,:}=(\boldsymbol{\mathcal{A}}*{V}_{i_2})_{i_1,:,:}
		\end{equation}
		and
		\begin{equation}\label{eq:fK2}
				\begin{split}
					(\boldsymbol{\mathcal{V}}_n*\boldsymbol{\mathcal{T}}_n+ \widetilde{\boldsymbol{\mathcal{V}}}_n)_{i_1,i_2,:,:}&= \sum_{k=1}^n(\boldsymbol{\mathcal{V}}_n)_{i_1,k,:,:}(\boldsymbol{\mathcal{T}}_n)_{k,i_2,:,:} \\
					& = (V_{i_2})_{i_1}\boldsymbol{\alpha}_{i_2}+(V_{i_2+1})_{i_1} \boldsymbol{\beta}_{i_2+1}+(V_{i_2-1})_{i_1} \\
					& = (V_{i_2}\times \boldsymbol{\alpha}_{i_2}+V_{i_2+1}\times \boldsymbol{\beta}_{i_2+1}+V_{i_2-1})_{i_1} .
				\end{split}
			\end{equation}
			The equality  between \eqref{eq:fK1} and \eqref{eq:fK2} follows using \eqref{eq:3rb} and proves \eqref{eq:tens3term1}.
			The remaining part of the theorem can be proved analogously. 
	\end{proof}
	If we $\ast$-multiply Eq.~\eqref{eq:tens3term1} by $\boldsymbol{\mathcal{W}_n}$ from the left we obtain the expression 

	\begin{equation*}\label{eq:tridiagproj}
		\boldsymbol{\mathcal{T}}_n=\boldsymbol{\mathcal{W}}_n*	\boldsymbol{\mathcal{A}}*\boldsymbol{\mathcal{V}}_n.
	\end{equation*}
	{The tensor $\boldsymbol{\mathcal{T}}_n$ satisfies a generalization of the \textit{matching moment property} which is  stated in Theorem \ref{thm:mmp}.}
	\begin{theorem}[Matching Moment Property]\label{thm:mmp}
		Let $\boldsymbol{\mathcal{A}}, V, W$ and $\boldsymbol{\mathcal{T}}_n$ be as described above, then
		\begin{equation*}
			W^D * (\boldsymbol{\mathcal{A}}^{k_*}) * V = E_1^D * (\boldsymbol{\mathcal{T}}_n)^{k_*}* E_1, \quad \text{ for } \quad k=0,\dots, 2n-1,
		\end{equation*}
		where $E_1 = \mathbf{e}_1 \otimes I_M$ and $\mathbf{e}_1$ is the first vector of the Euclidean base of $\mathbb{C}^{N \times N}$.
	\end{theorem}
	The proof of Theorem \ref{thm:mmp} can be found in Appendix \ref{appendix:mmp}.
	
	\subsection{Numerical properties} \label{sec:num_prop}
	The tensor $\boldsymbol{\mathcal{A}}$ is obtained by discretizing $\boldsymbol{\mathsf{A}}(t)$ and stores in $\boldsymbol{\mathcal{A}}_{k,l,:,:}$ the coefficients representing the $(k,l)$-th element of  $\boldsymbol{\mathsf{A}}(t)$.
	Different methods of discretization are possible.
	In this paper, following \cite{GisPoz22}, we discretize the interval $I$ obtaining the mesh
	\begin{equation}\label{eq:mesh}
		\tau_i = h (i-1) + a, \quad i = 1,\dots,M, \quad h = \frac{b-a}{M-1}. 
	\end{equation}
	For this mesh the discretization of $\boldsymbol{\mathsf{A}}(t) = \left[\boldsymbol{\mathsf{A}}_{k,\ell}(t) \right]_{k,\ell}^N$ is the tensor
	\begin{equation} \label{eq:block_matrix_formulation}
		\boldsymbol{\mathcal{A}}_{k,\ell,:,:} := \boldsymbol{\nu}_{k,\ell}, \quad k,\ell =1,\dots,N,
	\end{equation}
	where the matrices $\boldsymbol{\nu}_{k,\ell}\in \mathbb{C}^{M \times M}$ are lower triangular matrices defined as
	\begin{equation*}
		\left(\boldsymbol{\nu}_{k,\ell} \right)_{i,j} = \left\{\begin{array}{ll}
			\boldsymbol{\mathsf{A}}_{k,\ell}(\tau_i) h & i \geq j \\
			0 & i < j
		\end{array}\right . .
	\end{equation*}
	This discretization scheme has an accuracy of order $\mathcal{O}(h) = \mathcal{O}(1/M)$ and, indeed, in Section~\ref{sec:ex}, we show that when this discretization scheme is used, the approximation of the bilinear form of interest also has an accuracy of $\mathcal{O}(1/M)$. 
	
	The computational cost of Algorithm \ref{alg:Lanczos} depends on the chosen number of discretization points $M$ and on the number of iterations $n$.
	In this algorithm the dominant cost is the multiplication of a 4-mode tensor with a 3-mode tensor, i.e., $\boldsymbol{\mathcal{A}}\ast V_k$ and $W_k^D\ast \boldsymbol{\mathcal{A}}$.
	The worst case complexity of one such product is $\mathcal{O}(M^3 N^2)$, for a total cost of $\mathcal{O}(2 n M^3 N^2)$.
	However, since $\boldsymbol{\mathsf{A}}(t)$ is sparse in all practical applications, the computational cost can be much lower.
	For example, if there are $N_{\textrm{nz}}<N$ nonzeros on each column of $\boldsymbol{\mathsf{A}}(t)$, then the cost reduces to $\mathcal{O}(2 n M^3 N N_{\textrm{nz}} )$. It is important to note, moreover, that the term $M^3$ arises from the matrix-matrix multiplication between $V_k$ and the blocks in $\boldsymbol{\mathcal{A}}$. Since these blocks arise after a discretization strategy, it is likely that they will exhibit a particular structure that can be exploited for efficient computations. E.g., in the discretization used in this work, these blocks are lower triangular matrices for which the matrix-matrix multiplication has a cost of $\frac{M^3}{2}$.
	
	Finally, the storage cost of Algorithm \ref{alg:Lanczos} is three basisvectors $V_i$, three basisvectors $W_i$ and $3n-1$ nonzero elements of $\boldsymbol{\mathcal{T}_{n}}$, for a total of $\mathcal{O}(6M^2 N+ 3M^2 n)$.
	Only three basisvectors must be kept in memory thanks to the underlying three-term recurrence relation.
	
	Let us conclude this section observing that, as highlighted from the previous discussion, both, the computational cost and storage requirement depend strongly on the number $M$ of discretization points  used.
	For the discretization scheme described above, we expect that a large number of discretization points is required since its accuracy is $\mathcal{O}(1/M)$.
	This justifies the search for more accurate discretization schemes, for example Legendre polynomial approximation.
	However, other discretization schemes will not be discussed here since {they} are subject of future research and since the discretization scheme introduced above suffices to illustrate the potential of Algorithm \ref{alg:Lanczos}.
	
	\section{Breakdowns}\label{sec:break}
	If the matrix $\boldsymbol{\beta}_{k+1}$ is singular, then line 11 in Algorithm \ref{alg:Lanczos} cannot be performed and the algorithm breaks down.
	This breakdown issue is inherited from the (usual) non-Hermitian Lanczos algorithm; see, e.g., {\cite{WilBook65, ParTayLiu85, Gut92, Gut94b, FrGuNa93}}.
	There are two different {kinds} of breakdowns.
	The first one, the so-called \emph{lucky breakdown}, occurs when one of the Krylov-type subspaces $\mathcal{K}_k(\boldsymbol{\mathcal{A}},A)$ or $\mathcal{K}_k^{D}(B^D,\boldsymbol{\mathcal{A}})$ becomes invariant under $\star$-multiplication with $\mathcal{A}$ from the left or right, respectively.
	Suppose that $\mathcal{K}_k(\boldsymbol{\mathcal{A}},A)$ is an invariant subspace, this will result, in exact arithmetic, in $\widehat{V}_{k+1} = \boldsymbol{0}$ in Line 5 of Algorithm \ref{alg:Lanczos}. 
	In finite precision $\widehat{V}_{k+1}\in\mathbb{C}^{N\times M \times M}$ will never be exactly zero.
	Therefore, the Frobenius norm
	\begin{equation*}
		\Vert \widehat{V}_{k+1}\Vert_F := \sum_{i=1}^{N} \sum_{j=1}^{M} \sum_{{\ell=1}}^{M} \left\vert\left(\widehat{V}_{k+1}\right)_{i,j,{\ell}}\right\vert^2
	\end{equation*}
	is used to define the following criterion to detect a lucky breakdown:
	\begin{equation*}\label{eq:luckyCrit}
		\frac{\Vert \widehat{V}_{k+1}\Vert_F}{\Vert{V}_k\Vert_F} < \epsilon,
	\end{equation*}
	with $\epsilon<<1$, a user-defined threshold {close to machine precision}.
	The same applies to the case $\widehat{W}_{k+1}^D = \boldsymbol{0}$.
	The second kind of breakdown occurs when both $\widehat{V}_{k+1} \neq \boldsymbol{0}$ and $\widehat{W}_{k+1}^D \neq \boldsymbol{0}$, but $\boldsymbol{\beta}_{k+1}\in\mathbb{C}^{M\times M}$ is still singular, then the algorithm breaks down.  This case is known as {a} \emph{serious breakdown}.
	In numerical computation, the condition number of $\boldsymbol{\beta}_{k+1}$ is monitored to decide if Line 11 can be computed sufficiently {accurate}.
	A user-defined threshold $\epsilon_s>>1$ specifies an upper bound on the allowed condition number of $\boldsymbol{\beta}_{k+1}$.
	That is, if the ratio of its largest and smallest singular value is larger than $\epsilon_s$, i.e., $\sigma_{\max}(\boldsymbol{\beta}_{k+1})/\sigma_{\min}(\boldsymbol{\beta}_{k+1}) > \epsilon_s$, then the algorithm breaks down.
	Note that the choice of $\boldsymbol{\gamma}_{k+1}$ will influence the condition number of $\boldsymbol{\beta}_{k+1}.$

	In the usual non-Hermitian Lanczos algorithm, {a} serious breakdown can be treated by using a so-called \emph{look-ahead} strategy; see, e.g., \cite{Tay82, ParTayLiu85, FrGuNa93, Gut92, Gut94b, BreRedSad91, Brezinski1992}. Connection between serious-breakdowns, (formal) orthogonal polynomials, and matching moment property can be found in \cite{Draux96, PozPra19}.
	If needed, an analogous look-ahead strategy may be implemented for the tensor Lanczos algorithm. 
	At the moment, an easier strategy to deal with serious breakdowns is to reformulate the problem so to change the input vectors $\mathbf{v}, \mathbf{w}$.
	For instance, when $\mathbf{w}= \boldsymbol{e}_i$ and $\mathbf{v}=\boldsymbol{e}_i$, a serious breakdown is likely to happen due to the sparsity of $\boldsymbol{\mathcal{A}}$.
	However, we can rewrite the approximation of the time-ordered exponential $\boldsymbol{\mathsf{U}}(t)$ as
	$$\boldsymbol{e}_i^H\boldsymbol{\mathsf{U}}(t)\boldsymbol{e}_j = (\boldsymbol{e} + \boldsymbol{e}_i)^H \boldsymbol{\mathsf{U}}(t) \boldsymbol{e}_j - \boldsymbol{e}^H \boldsymbol{\mathsf{U}}(t) \boldsymbol{e}_j,
	$$
	with $\boldsymbol{e} = (1, \dots, 1)^H$. Then one can approximate $(\boldsymbol{e} + \boldsymbol{e}_i)^H \boldsymbol{\mathsf{U}}(t) \, \boldsymbol{e}_j$ and $\boldsymbol{e}^H \boldsymbol{\mathsf{U}}(t) \, \boldsymbol{e}_j$ separately, which are less likely going to have a breakdown thanks to the fact that $\boldsymbol{e}$ is a full vector; see, e.g., \cite[Section 7.3]{GolMeuBook10} and \cite{GisPoz22}.

\section{Numerical examples}\label{sec:ex}
Let us consider the following {smooth matrix-valued} function defined on a real interval $I = [a,b]$:
\begin{equation*}
	\boldsymbol{\mathsf{A}}(t): I \subset \mathbb{C}  \rightarrow \mathbb{C}^{N \times N}.
\end{equation*}
As anticipated in the Introduction, the time-ordered exponential of $\boldsymbol{\mathsf{A}}(t)$ is the unique matrix-valued function $\boldsymbol{\mathsf{U}}(t) \in \mathbb{C}^{N\times N}$ defined on  $I$ that is the solution of the system of linear ordinary differential equations
\begin{equation*}
	\frac{d}{dt}\boldsymbol{\mathsf{U}}(t) = \boldsymbol{\mathsf{A}}(t) \boldsymbol{\mathsf{U}}(t), \quad \boldsymbol{\mathsf{U}}(a)=I_N, \quad t \in I,
\end{equation*}
see \cite{dyson1952}.
In this section, we aim to approximate the bilinear form $\mathbf{w}^H \boldsymbol{\mathsf{U}}(t) \mathbf{v}$ by using the tensor non-Hermitian Lanczos algorithm.
If the matrix $\boldsymbol{\mathsf{A}}$ is so that $\boldsymbol{\mathsf{A}}(\tau_1)\boldsymbol{\mathsf{A}}(\tau_2)-\boldsymbol{\mathsf{A}}(\tau_2)\boldsymbol{\mathsf{A}}(\tau_1)=0$ for all $\tau_1,\tau_2 \in I$, then $\boldsymbol{\mathsf{U}}(t)=\exp\left(\int_s^{t} \boldsymbol{\mathsf{A}}(\tau)\, \textrm{d}\tau\right).$
Unfortunately, $\boldsymbol{\mathsf{U}}(t)$ -- and the related bilinear forms -- cannot be expressed by an analogous simple form in the general case.
Indeed, even for small matrices, $\boldsymbol{\mathsf{U}}(t)$ may be given by complicated special functions \cite{Xie2010,Hortacsu2018}.

A new approach for the approximation of a time-ordered exponential {bilinear form} was introduced in \cite{GisPozInv19,GisPoz21,GisPoz22} and it is based on $\star$-Lanczos, which is a symbolic algorithm.
This method is able to approximate the bilinear form 
$$\mathbf{w}^H \boldsymbol{\mathsf{U}}(t) \mathbf{v}, \quad t \in I$$
for the given vectors $\mathbf{w},\mathbf{v}$, with $\mathbf{w}^H, \mathbf{v} \neq 0$. 
The matrices $\boldsymbol{\alpha}_1,\dots,\boldsymbol{\alpha}_n$, $\boldsymbol{\beta}_2,\dots,\boldsymbol{\beta}_n$, and $\boldsymbol{\gamma}_2,\dots,\boldsymbol{\gamma}_n$, which compose the $4$-mode tensor $\boldsymbol{\mathcal{T}}_n$ in \eqref{eq:Tcoeff}, are obtained by running $n$ iterations of Algorithm \ref{alg:Lanczos} {with}, as inputs, the $4$-mode tensor $\boldsymbol{\mathcal{A}}$ in \eqref{eq:block_matrix_formulation} and the vectors $\mathbf{v},\mathbf{w}$.

Sampling the true solution $\mathbf{w}^H \boldsymbol{\mathsf{U}}(t) \mathbf{v}$ on the discretization nodes $\tau_i$ gives the vector $\mathbf{\hat{s}}$ defined as
$$ {\mathbf{\hat{s}}:= \begin{bmatrix}
		\mathbf{w}^H \boldsymbol{\mathsf{U}}(\tau_1) \mathbf{v} & \mathbf{w}^H \boldsymbol{\mathsf{U}}(\tau_2) \mathbf{v} & \dots & \mathbf{w}^H \boldsymbol{\mathsf{U}}(\tau_M) \mathbf{v}
	\end{bmatrix}^\top.} $$
Exploiting the results described in \cite{GisPoz22}, the sampled solution vector $\mathbf{\hat{s}}$ can be approximated by  
\begin{equation}\label{eq:odeappro}
	{\mathbf{s}_n} := \frac{1}{h}\left(\boldsymbol{\theta} \times \left(R_\ast(\boldsymbol{\mathcal{T}}_n)\right)_{1,1,:,:}\right) \mathbf{e}_1 \approx  \mathbf{\hat{s}},
\end{equation}
where  $R_\ast$ is the $\ast$-resolvent {, i.e., the tensor} 
\begin{equation*}\label{eq:starres}
	R_\ast(\boldsymbol{\mathcal{T}}_n) := \boldsymbol{\mathcal{I}}_\ast + \sum_{k=1}^{\infty} \left( \boldsymbol{\mathcal{T}}_n \right)^{k_\ast},
	\end{equation*}
and
\begin{equation*}\label{eq:heavisideDiscr}
	\boldsymbol{\theta} :=   h \begin{bmatrix}
		1 & 0&  0 &\dots & 0\\
		1 & 1 & 0 & \dots & 0\\
		\vdots & \vdots & \ddots & \ddots & \vdots\\
		1 & 1 & \dots & 1 & 0\\
		1 & 1 & \dots & 1 & 1
	\end{bmatrix} \in \mathbb{C}^{M\times M}.
\end{equation*}
{Overall, the accuracy of the approximation in \eqref{eq:odeappro} can not be better than $\mathcal{O}(h)$. This is due to the fact that, as explained in \cite{GisPoz22}, the discretization \eqref{eq:block_matrix_formulation} is based on a rectangular quadrature rule.}
Finally, using the Path-sum method \cite{Giscard2015} we get the following explicit expression for the $\ast$-resolvent in terms of a continued fraction
\begin{equation} \label{eq:star_resolvent}
    \begin{split}
    &R_{\ast}(\boldsymbol{\mathcal{T}_{n}})_{1,1,:,:} \hspace{-1pt} =  \\ & \hspace{-1pt} \left( \hspace{-2pt} \widetilde{\boldsymbol{\alpha}}_1 \hspace{-2pt} - \hspace{-2pt} \boldsymbol{\beta}_2 \left( \hspace{-2pt} \widetilde{\boldsymbol{\alpha}}_2 \hspace{-2pt} - \hspace{-2pt} \boldsymbol{\beta}_3 \left( \cdots \boldsymbol{\beta}_{n-1} \widetilde{\boldsymbol{\alpha}}_n^{-1}  \boldsymbol{\gamma}_{n-1} \cdots \right)^{-1} \boldsymbol{\gamma}_3 \right)^{ -1}  \boldsymbol{\gamma}_2 \right)^{-1},
    \end{split}
\end{equation}
with $\widetilde{\boldsymbol{\alpha}}_i = I_{{M}} - \boldsymbol{\alpha}_i$.
{Equation \eqref{eq:star_resolvent} is computed from the most inner term moving outward, where the inversion operation is performed using the \texttt{backslash} operator in Matlab\textsuperscript{\textregistered}.}
{Note that the $\ast$-resolvent and all inverses appearing in \eqref{eq:star_resolvent} are expected to exist for $h$ small enough, since their continuous counterparts exist under certain regularity conditions on $\boldsymbol{\mathsf{A}}(t)$; see \cite{GisPozInv19,GisPoz21}.}

The rest of the section is structured as follows. Section \ref{sec:errorMeasures} describes the measures that will be used to quantify the errors of the final solution and of the computed Krylov bases.
In Section \ref{sec:fexp_GP} two examples are discussed for which an analytical solution is available.
This allows us to compare the approximation to an exact solution and to show that it converges with the expected rate of convergence.
Small scale examples from NMR spectroscopy are discussed in Section \ref{sec:exp_NMR}. Finally, in Section \ref{sec:TT_ranks}, we analyse the approximability of the previously considered tensors by the Tensor Train representation.

\subsection{Error measures}\label{sec:errorMeasures}
In this section we define a series of error measures which {quantify} the quality of the generated biorthogonal bases and the accuracy of the approximation \eqref{eq:odeappro}.
These measures use the Frobenius norm, which, for a 4-modes tensor, is defined as
\begin{equation*}
	\Vert \boldsymbol{\mathcal{A}}\Vert_F := \sum_{i=1}^{N} \sum_{j=1}^{N} \sum_{k=1}^{M} \sum_{l=1}^{M} \left\vert\left(\boldsymbol{\mathcal{A}}\right)_{i,j,k,l}\right\vert^2.
\end{equation*}
The main goal is to analyze the rate of convergence as the number of discretization points $M$ is increased.
To stress the dependence on $M$ of computed quantities we use the superscript ``${(M)}$''.\\
{A generalization of the usual error measures for Krylov subspace methods are used. As a measure} for the biorthonormality of the bases $\boldsymbol{\mathcal{V}}_{n} \in \mathbb{C}^{N\times n \times M \times M }$ and  $\boldsymbol{\mathcal{W}}_{n}\in \mathbb{C}^{n\times N \times M \times M }$ generated by $n$ steps of the algorithm, we use
\begin{equation*}
	\textrm{err}_\textrm{o} := \frac{\Vert \boldsymbol{\mathcal{W}}_{n}^{(M)}*\boldsymbol{\mathcal{V}}_{n}^{(M)} - \boldsymbol{\mathcal{I}}_*\Vert_F}{\max (\Vert\boldsymbol{\mathcal{V}}_{n}^{(M)} \Vert_F,\Vert\boldsymbol{\mathcal{W}}_{n}^{(M)} \Vert_F)}.
\end{equation*}
{For a robust algorithm, it is paramount that the term $\max (\Vert\boldsymbol{\mathcal{V}}_{n}^{(M)} \Vert_F,\Vert\boldsymbol{\mathcal{W}}_{n}^{(M)} \Vert_F)$ remains small as $n$ increases. 
This can be obtained by employing an appropriate strategy to rescale the basisvectors in $\boldsymbol{\mathcal{V}}_{n}^{(M)}$ and $\boldsymbol{\mathcal{W}}_{n}^{(M)}$, i.e., by $\gamma_{k+1}$ in Algorithm \ref{alg:Lanczos}.
In this section we will choose $\gamma_{k+1}=I_M$ for all $k$, i.e., no rescaling. 
An effective rescaling strategy can improve on the numerical results reported below, but developing such a strategy is subject to future research.}

{To measure the} quality of the recurrences \eqref{eq:recurrences},
we use
\begin{align*}
	\textrm{err}_\textrm{V} := \frac{\Vert \boldsymbol{\mathcal{A}}*\boldsymbol{\mathcal{V}}_{n}^{(M)} - \boldsymbol{\mathcal{V}}_{n}^{(M)}*\boldsymbol{\mathcal{T}}_{n}^{(M)} -  \widetilde{\boldsymbol{\mathcal{V}}}_{n}^{(M)}\Vert_F}{\max(\Vert \boldsymbol{\mathcal{A}}*\boldsymbol{\mathcal{V}}_{n}^{(M)} \Vert_F, \Vert\boldsymbol{\mathcal{V}}_{n}^{(M)}*\boldsymbol{\mathcal{T}}_{n}^{(M)} + \widetilde{\boldsymbol{\mathcal{V}}}_{n}^{(M)}\Vert_F)},\\
	\textrm{err}_\textrm{W} :=\frac{\Vert \boldsymbol{\mathcal{W}}_{n}^{(M)}*\boldsymbol{\mathcal{A}} - \boldsymbol{\mathcal{T}}_{n}^{(M)}*\boldsymbol{\mathcal{W}}_{n}^{(M)} - \widetilde{\boldsymbol{\mathcal{W}}}_{n}^{(M)}\Vert_F}{\max(\Vert \boldsymbol{\mathcal{W}}_{n}^{(M)}*\boldsymbol{\mathcal{A}} \Vert_F, \Vert \boldsymbol{\mathcal{T}}_{n}^{(M)}*\boldsymbol{\mathcal{W}}_{n}^{(M)} + \widetilde{\boldsymbol{\mathcal{W}}}_n^{(M)}\Vert_F)}.
\end{align*}
As a measure for the Matching Moment Property, see Theorem \ref{thm:mmp}, we use
\begin{equation*}
	\textrm{err}_\textrm{M}(k) :=  \frac{\Vert W^D * (\boldsymbol{\mathcal{A}}^{k_*}) * V - E_1^D * (\boldsymbol{\mathcal{T}}_{n}^{(M)})^{k_*}* E_1 \Vert_F}{\max(\Vert W^D * (\boldsymbol{\mathcal{A}}^{k_*}) * V \Vert_F, \Vert E_1^D * (\boldsymbol{\mathcal{T}}_{n}^{(M)})^{k_*}* E_1\Vert_F) },
\end{equation*}
which should be close to zero for $k=0,\dots, 2n-1$.

Finally, to quantify the quality of the solution, we consider as error measure for \eqref{eq:odeappro} the quantity
\begin{equation*}
	\textrm{err}_{\textrm{sol}} := \frac{\Vert \mathbf{\hat{s}}-\mathbf{s}^{(M)}_n \Vert_2}{\Vert \mathbf{\hat{s}} \Vert_2},
\end{equation*}
where, if no analytic expression is available, an approximation of $\mathbf{s}$ is obtained by using \verb*|ode45| in Matlab\textsuperscript{\textregistered}. In the formula above, $\Vert \cdot \Vert_2$ stands for the usual Euclidean norm.
The rate at which $\textrm{err}_{\textrm{sol}}$ decreases as $M$ increases is expected to be $\mathcal{O}(h)= \mathcal{O}(1/M)$, i.e., the accuracy of the discretization used here.

\subsection{Proof of concept}\label{sec:fexp_GP}
As a proof of concept, we test our proposal on two problems {which} originally appeared in \cite{GisPoz22}. 
In both experiments {a discretization with $M$ points is used and} we run $n$ iterations of Algorithm \ref{alg:Lanczos} {with $\gamma_{k+1} \equiv I_M$}.
	This produces the tensor $\boldsymbol{\mathcal{T}}_{n}^{(M)}$ defined in \eqref{eq:Tcoeff} with coefficients $\boldsymbol{\alpha_{1}}^{(M)},\dots, \boldsymbol{\alpha_{n}}^{(M)}$ and $\boldsymbol{\beta_{2}}^{(M)},\dots, \boldsymbol{\beta_{n}}^{(M)}$, depending on $M$.
	For the two experiments considered here the result of the $\star$-Lanczos algorithm \cite{GisPoz22} is known.
	The coefficients resulting from this latter algorithm are bivariate functions $\alpha_{1}(t,s),\dots,\alpha_{n}(t,s)$ and $\beta_{2}(t,s),\dots,\beta_{n}(t,s)$, because $\star$-Lanczos is a symbolic algorithm.
	The tensor Lanczos algorithm is a discretization of the $\star$-Lanczos algorithm, which means that $\boldsymbol{\alpha_{i}}^{(M)}$ and $\boldsymbol{\beta_{i}}^{(M)}$ can be seen as discretizations of the functions $\alpha_{i}(t,s)$ and $\beta_{i}(t,s)$, respectively.\\
	Consider the evaluation of these functions on the mesh $\tau_i$:
	\begin{equation*}
		\boldsymbol{ \hat{\alpha}}_i := \begin{bmatrix}
			\alpha(\tau_i,\tau_j)
		\end{bmatrix}_{i,j=1}^{M}, \qquad \boldsymbol{ \hat{\beta}}_i := \begin{bmatrix}
			\beta(\tau_i,\tau_j)
		\end{bmatrix}_{i,j=1}^{M},
	\end{equation*}
	then, we can define the errors $\frac{\Vert \boldsymbol{\hat{\alpha}}_{i} - \boldsymbol{\alpha}_{i}^{(M)}\Vert_2}{\Vert \boldsymbol{\hat{\alpha}}_{i} \Vert_2}$, $i=1,\dots,n$, and $\frac{\Vert \boldsymbol{\hat{\beta}}_{i} - \boldsymbol{\beta}_{i}^{(M)}\Vert_2}{\Vert\boldsymbol{ \hat{\beta}}_{i} \Vert_2}$, $i=2^,\dots,n$.
	These errors will be used as {a} measure {for} the accuracy of the computed tensor $\boldsymbol{\mathcal{T}}_{n}^{(M)}$.
	The number of iterations $n$ is chosen equal to the problem size $N$, which allows us to compare all the available functions $\alpha_{1}(t,s),\dots, \alpha_N(t,s)$ and $\beta_{2}(t,s),\dots, \beta_N(t,s)$ with the elements in $\boldsymbol{\mathcal{T}}_{N}^{(M)}$ and track the convergence rate with $M$ of the latter.

\subsubsection{Time independent matrix}
Consider a constant matrix and starting vectors
\begin{equation*}
	{\boldsymbol{\mathsf{A}}}(t) = \begin{bmatrix}
		-1 & 1 & 1\\
		1 & 0 & 1\\
		1 & 1 & -1 
	\end{bmatrix},\qquad \mathbf{v},\mathbf{w} =  \begin{bmatrix}
		1\\
		0\\
		0
	\end{bmatrix},
\end{equation*}
and the interval $I=[0,1]$.
The inputs {of} Algorithm \ref{alg:Lanczos} are the starting  hypervectors $\mathbf{v}\otimes I_M, \mathbf{w}\otimes I_M$ and the tensor $\boldsymbol{\mathcal{A}}$  whose components $\boldsymbol{\mathcal{A}}_{i_1,i_2,:,:}$ are defined as
\begin{equation*}
	\boldsymbol{\mathcal{A}}_{{i_1,i_2},:,:} = \begin{cases}
		\boldsymbol{\theta}, &\text{ if } {i_1=i_2} = 1 \text{ or } {i_1=i_2}=3\\
		\boldsymbol{0}, &\text{ if } {i_1=i_2}=2 \\
		-\boldsymbol{\theta}, &\text{ otherwise} 
	\end{cases},
\end{equation*}
where $\boldsymbol{0}\in\mathbb{R}^{M\times M}$ is the null matrix.
The tensor $\boldsymbol{\mathcal{A}}$ is obtained by sampling the matrix-valued function $\boldsymbol{\mathsf{A}}(t)$ on the $M$ point mesh \eqref{eq:mesh} and following the definition in \eqref{eq:block_matrix_formulation}.
The output for $n=N=3$ iterations is $\boldsymbol{\mathcal{T}}_{N}^{(M)}$.
Table \ref{table:ex1_KrylovErrors} reports the Krylov error measures, which behave as expected for the recurrence measures.
The loss of biorthogonality observed for increasing values of $M$ is presumably due to the fact that no rescaling is used in the algorithm.

\begin{table}[!ht]
	\centering
	\begin{tabular}{l|lll}
		$M$                       & 10         & 100        & 1000       \\ \hline
		 \hline
		$\textrm{err}_\textrm{o}$ & 2.212e-15 & 4.127e-13 & 5.389e-11  \\
		$\textrm{err}_\textrm{V}$ & 2.234e-16 & 6.356e-15 & 4.327e-14 \\
		$\textrm{err}_\textrm{W}$ & 0          & 0          & 0          \\
	\end{tabular}
	\caption{Krylov error measures for time independent matrix.}
	\label{table:ex1_KrylovErrors}
\end{table}

On the other hand, this loss of biorthogonality does not compromise the moment matching capabilities of $\boldsymbol{\mathcal{T}}_{N}^{(M)}$, as it becomes evident from Table \ref{table:ex1_moment} where we report $\textrm{err}_\textrm{M}(k)$ for $k\leq 5=2n-1$.

\begin{table}[!ht]
	\centering
	\begin{tabular}{l|lll}
		\backslashbox{$k$}{$M$} & 10         & 100        & 1000       \\ \hline
		 \hline
		3                       & 6.078e-17 & 9.464e-17  & 9.580e-16 \\
		4                       & 6.611e-17 & 2.806e-16 & 1.199e-15 \\
		5                       & 9.840e-17   & 2.558e-16 & 2.844e-15
	\end{tabular}
	\caption{Measure for moment matching $\textrm{err}_\textrm{M}(k)$ for time independent matrix. {Entries for $k=0,1,2$ are omitted since they are equal to zero.}}
	\label{table:ex1_moment}
\end{table}

Moreover, as the values reported in Table \ref{table:ex1_coeffs} confirm, we can observe that the elements $\boldsymbol{\beta}_i$ converge at the expected rate of $\mathcal{O}(1/M)$.
The elements $\boldsymbol{\alpha}_i$ are very accurate for $M=10$ whereas for larger $M$, the accuracy of $\boldsymbol{\alpha}_i$ decreases: this decrease is, presumably, the result of error propagation in the numerical algorithm.
{This} error is still smaller than the expected order of $\mathcal{O}(1/M)$.

\begin{table}[!ht]
	\centering
	\begin{tabular}{lllllll}
		$M$                      & \multicolumn{2}{l}{10}  & \multicolumn{2}{l}{100} & \multicolumn{2}{l}{1000} \\ \hline
		 \hline
		\multicolumn{1}{l|}{$i$} & $\boldsymbol{\alpha}_i$   & $\boldsymbol{\beta}_i$   & $\boldsymbol{\alpha}_i$   & $\boldsymbol{\beta}_i$     & $\boldsymbol{\alpha}_i$   & $\boldsymbol{\beta}_i$     \\ \cline{2-7} 
		\multicolumn{1}{l|}{1}   & 0          & /          & 0          & /          & 0           & /          \\
		\multicolumn{1}{l|}{2}   & 7.639e-16 & 2.365e-01 & 3.867e-13 & 2.250e-02 & 1.555e-10  & 2.240e-03 \\
		\multicolumn{1}{l|}{3}   & 2.875e-15 & 2.365e-01 & 1.283e-11 & 2.250e-02 & 5.118e-08  & 2.240e-03
	\end{tabular}
	\caption{Error on the elements of the tridiagonal tensor for time independent matrix.}
	\label{table:ex1_coeffs}
\end{table}

Moreover, in this particular case, the analytical solution to the ODE is known; see \cite{GisPoz22}:
\begin{equation*}
	\mathbf{\hat{s}} = \begin{bmatrix}
		\left(\exp(A\tau_1) \right)_{11} & \left(\exp(A\tau_2) \right)_{11} & \dots & \left(\exp(A\tau_M) \right)_{11}
	\end{bmatrix}^\top,
\end{equation*}
with $\left(\exp(At) \right)_{11} = -\frac{1}{2} \sinh(2t) + \frac{1}{2}\cosh(2t) + \frac{1}{2} \cosh(\sqrt{2}t)$.
Hence, it is possible to compare this exact solution with \eqref{eq:odeappro}. Note that for $n=3$ the $\ast$-resolvent is given by the continued fraction:
\begin{equation*}
	R_{\ast}(\boldsymbol{\mathcal{T}_{3}})_{1,1,:,:} = \left(I_M - \boldsymbol{\alpha}_1 - \left(I_M - \boldsymbol{\alpha}_{2} - \left(I_M - \boldsymbol{\alpha}_{3}\right)^{-1} \boldsymbol{\beta}_3 \right)^{ -1} \boldsymbol{\beta}_2 \right)^{-1}.
\end{equation*}
Table \ref{table:ex1_errsol} shows the error measure $\textrm{err}_\textrm{sol}$ for increasing $M$, which convergences at the expected rate $\mathcal{O}(1/M)$.

\begin{table}[!ht]
	\centering
	\begin{tabular}{l|llll}
		$M$                                 & 10         & 100         & 1000          \\ \hline
		 \hline
		$\textrm{err}_\textrm{sol}$  & 8.230e-02 & 7.019e-03 & 6.918e-04 
	\end{tabular}
	\caption{Error of approximation to the quantity of interest $\mathbf{w}^H \boldsymbol{\mathsf{U}}(t) \mathbf{v}$ for time independent matrix.}
	\label{table:ex1_errsol}
\end{table}

\subsubsection{Time dependent matrix}
Consider the time-dependent matrix
\begin{equation*}
	\boldsymbol{\mathsf{\tilde{A}}}(t) = \begin{bmatrix}
		\cos(t) & 0 & 1 & 2 & 1\\
		0 & \cos(t)-t & 1-3t & t & 0\\
		0 & t & 2t+\cos(t)  & 0 & 0\\
		0 & 1 & 2t+1 & t + \cos(t) & t\\
		t & -t-1 & -6t-1 & 1-2t & \cos(t)-2t
	\end{bmatrix},
\end{equation*}
the starting vectors $\mathbf{v} = \mathbf{w} = \begin{bmatrix}
	1&	0&	0& 0 & 0
\end{bmatrix}^\top$, and the interval $I=[10^{-4},1]$.
As it becomes apparent {from} the results {reported} in Table \ref{table:ex2_KrylovErrors}, for this particular experiment, we obtain that the Krylov error measures and the recurrence measures are small whereas the biorthogonality measure is large.
\begin{table}[!ht]
	\centering
	\begin{tabular}{l|lll}
		$M$                         & 10       & 100      & 1000           \\ \hline
		 \hline
		$\textrm{err}_{\textrm{o}}$ & 1.069e-01  & 7.866e-01  & 8.645e-01 \\
		$\textrm{err}_{\textrm{V}}$ &  5.262e-16  & 1.557e-14 & 8.780e-16 \\
		$\textrm{err}_{\textrm{W}}$ & 7.062e-18  & 1.077e-17 & 1.041e-17 \\
	\end{tabular}
	\caption{Krylov error measures for time dependent matrix.}
	\label{table:ex2_KrylovErrors}
\end{table}
The loss of orthogonality is an inherent feature of Lanczos-like algorithms, and it does not necessarily {compromise algorithm's capability to produce an approximation to the bilinear form}. {Indeed, Table \ref{table:ex2_moments} confirms that the} matching moment property is not affected by the loss of $\ast$-biorthogonality. 

\begin{table}[!ht]
	\centering
	\begin{tabular}{l|lll}
		\backslashbox{$k$}{$M$} & 10         & 100        & 1000       \\ \hline
		 \hline
		3                       & 5.439e-17 & 1.375e-16  & 4.226e-16 \\
		4                       & 1.357e-16 & 1.923e-16 & 4.514e-16 \\
		5                       & 1.231e-16 & 2.370e-16 & 5.053e-15 \\
		6                       & 1.443e-16 & 3.069e-16 & 7.870e-15 \\
		7                       & 1.479e-16 & 3.553e-16  & 2.704e-14 \\
		8                       & 1.654e-16 & 3.271e-16 & 5.282e-14 \\
		9                       & 1.530e-16 & 3.775e-16 & 3.022e-13
	\end{tabular}
	\caption{Measure for moment matching $\textrm{err}_\textrm{M}(k)$ for time dependent matrix. {Entries for $k=0,1,2$ are omitted since they are equal to zero.}}
	\label{table:ex2_moments}
\end{table}

On the other hand, the results presented in Table \ref{table:ex2_coeffs},  where we report the error measures for the coefficients computed by Algorithm \ref{alg:Lanczos}, show that the loss of $\ast$-biorthogonality of the computed bases has a limited impact for the convergence of the algorithm to the solution. Indeed, in this case, for  all $\boldsymbol{\beta}_i$ the expected convergence rate is observed whereas, for $\boldsymbol{\alpha_{i}}$, only $i=1,2,3$ show the expected decrease in the error measure.

\begin{table}[!ht]
	\centering
	\begin{tabular}{lllllll}
		$M$                      & \multicolumn{2}{l}{10}  & \multicolumn{2}{l}{100} & \multicolumn{2}{l}{1000} \\ \hline
		 \hline
		\multicolumn{1}{l|}{$i$} & $\boldsymbol{\alpha}_i$   & $\boldsymbol{\beta}_i$  & $\boldsymbol{\alpha}_i$   & $\boldsymbol{\beta}_i$  &$\boldsymbol{\alpha}_i$   & $\boldsymbol{\beta}_i$     \\ \cline{2-7} 
		\multicolumn{1}{l|}{1}   & 0          & /          & 0          & /          & 0           & /          \\
		\multicolumn{1}{l|}{2}   & 6.986e-13 & 2.560e-01  & 2.139e-10  & 2.388e-02   & 4.751e-09   &  2.373e-03 \\
		\multicolumn{1}{l|}{3}   & 2.034e-02 & 4.057e-01 & 2.098e-03 & 3.493e-02 & 6.507e-03   &  3.445e-03 \\
		\multicolumn{1}{l|}{4}   & 6.473e-02 &  2.858e-01 & 6.679e-03 & 2.842e-02  & 9.859e-01  & 2.845e-03 \\
		\multicolumn{1}{l|}{5}   & 1.452e-01 & 1.754e-01 & 2.551e+01 & 1.695e-02 &  8.365e+05 & 1.840e-03\\
	\end{tabular}
	\caption{Error on the elements of the tridiagonal tensor for time dependent matrix.}
	\label{table:ex2_coeffs}
\end{table}

Table \ref{table:ex2_errsol} shows that the approximation to the quantity of interest converges at the rate $\mathcal{O}(1/M)$.
Hence, the loss of biorthogonality and the inaccurate coefficients of the tridiagonal tensor did not compromise this approximation.
\begin{table}[!ht]
	\centering
	\begin{tabular}{l|llll}
		$M$                                 & 10         & 100         & 1000          \\ \hline
		 \hline
		$\textrm{err}_\textrm{sol}$  & 2.360e-01 & 2.257e-02 & 2.404e-03 
	\end{tabular}
	\caption{Error of approximation to the quantity of interest $\mathbf{w}^H \boldsymbol{\mathsf{U}}(t) \mathbf{v}$ for time dependent matrix.}
	\label{table:ex2_errsol}
\end{table}

\subsection{NMR experiments}\label{sec:exp_NMR}
Nuclear magnetic resonance (NMR) spectroscopy studies the structure and dynamics of molecules by looking at nuclear spins \cite{Ho99,Le08}.
Computer simulations of NMR experiments are important because they can improve the design and analysis of laboratory experiments \cite{SmPaWiGe92}.
In this section, three small, realistic examples arising from NMR spectroscopy \cite{BalBon22} are discussed.
The ODE that governs the dynamics of nuclear spins during NMR spectroscopy is the Schr\"odinger equation
\begin{equation*}
	\frac{d}{dt} \phi(t) = -\imath 2\pi H(t) \phi(t), \quad t\in\left[0,\tau_{\textrm{exp}}\right]
\end{equation*}
where $H(t)$ is the so-called Hamiltonian, $\phi(t)$ the wave function, $\tau_{\textrm{exp}}$ the duration of the experiment and $\imath = \sqrt{-1}$.
The size of the Hamiltonian is related to the number of nuclear spins present in the system, for $l$ nuclear spins $H(t)$ is of the size $2^l\times 2^l$.
Hence, $H(t)$ grows exponentially with the number of spins, but it is a sparse matrix, making it an ideal candidate for a Lanczos-like algorithm.\\
The experiments discussed in this section use $M=500$ discretization points, because of memory constraints. It is important to note that such memory constraints {may} be overcome using the Tensor Train approximation presented  in Section~\ref{sec:TT_ranks}. 
The number of iterations of the tensor Lanczos algorithm is chosen to obtain the maximal attainable accuracy, which is determined by the discretization scheme with $M=500$. 
Since the discretization used here has an accuracy of $\mathcal{O}(1/M)$, the smallest number of iterations $n$ such that $\textrm{err}_{\textrm{sol}}$ is of order $10^{-3}$ suffices.
Choosing larger $n$ will not decrease $\textrm{err}_{\textrm{sol}}$ further for a fixed $M$ and will increase the computational cost.

\subsubsection{Experiment 1: weak coupling}
Consider four nuclear spins with heteronuclear dipolar couplings. In this framework, the Hamiltonian for a magic angle spinning (MAS) experiment \cite{HaSp98} is the diagonal matrix
\begin{align*}
	H(t) = \textrm{diag}\left[\{f_k(t)\}_{k=1}^{16}\right],\quad  f_k(t) = \alpha_k + \beta_k \cos(2\pi \nu  t) + \gamma_k \cos(4\pi \nu t),
\end{align*}
with $\alpha_k,\beta_k,\gamma_k\in\mathbb{R}$ and $\nu=$1e+4.
The diagonal matrix $A(t)=-\imath 2 \pi H(t)$ commutes with itself at all times and thus the solution $U(t)$ can be computed
\begin{align*}
	U(t) &= \textrm{diag}\left[\left\{\exp\left(-\imath \alpha_k t -\imath \frac{\beta_k}{2\pi\nu}\sin(2\pi \nu t)-\imath \frac{\gamma_k}{4\pi \nu} \sin(4\pi \nu t)\right) \right\}_{k=1}^{16}\right].
\end{align*}
The starting vectors are chosen to excite and measure the lowest oscillatory components in $U(t)$: $\mathbf{w} =\mathbf{v} = \begin{bmatrix}
	0 & 1 & 1 & 0 & 1 & 1& \dots & 0& 1& 1
\end{bmatrix}^\top$.
A typical experiment would run for a time of the order $1e-2$ seconds.
Since the problem is (highly) oscillatory and the current discretization requires many points to {accurately compute a solution}, we choose to restrict the experiment time to $\tau_{\textrm{exp}}=5e-5$.
This is a valid approach since the total time interval of $1e-2$ can be split into subintervals of length $5e-5$ and the solutions on the subintervals can be combined to obtain the solution on the whole interval.

Algorithm \ref{alg:Lanczos} is run for $n = 3$ iterations and the corresponding Krylov error measures are shown in Table \ref{table:exA_KrylovErrors}.
A first observation concerns the fact that going from $M=5$ to $M=50$ a large decrease of the biorthogonality measure is observed. {This is due to the fact that when discretizing with fewer discretization points, e.g., $M=5$, the original matrix in the ODE $-\imath 2 \pi H(t)$ is translated into a simpler (and inaccurate) discretized input, for which} the tensor Lanczos iteration converges fast.
The discretization $M=50$ represents the original input better, as is suggested by the stagnation of $\textrm{err}_{\textrm{o}}$ going from $M=50$ to $M=500$.
\begin{table}[!ht]
	\centering
	\begin{tabular}{l|lll}
		$M$                         & 5       & 50      & 500           \\ \hline
		\hline
		$\textrm{err}_{\textrm{o}}$ & 1.435e-03 & 1.422e-08   & 1.298e-08    \\
		$\textrm{err}_{\textrm{V}}$ & 2.947e-16  & 8.804e-14 & 7.063e-11  \\
		$\textrm{err}_{\textrm{W}}$ & 4.015e-17  & 2.553e-17 & 2.772e-17 \\
		$\textrm{err}_\textrm{sol}$  & 1.810e-01 & 2.024e-02 & 2.076e-03
	\end{tabular}
	\caption{Error measures for Experiment 1.}
	\label{table:exA_KrylovErrors}
\end{table}
The error $\textrm{err}_\textrm{sol}$ is computed using the analytical solution $\mathbf{w}^H \boldsymbol{\mathsf{U}}(t) \mathbf{v}$ evaluated in the discretization points and decays at the expected rate $\mathcal{O}(1/M)$.
Figure \ref{fig:exA_plotS} shows $\mathbf{\hat{s}}$ and the approximation $\mathbf{s}^{(M)}_n$, which clearly converges for increasing $M$.

\begin{figure}[!ht]
	\centering
	\setlength\figureheight{5cm}
	\setlength\figurewidth{\textwidth}
	\input{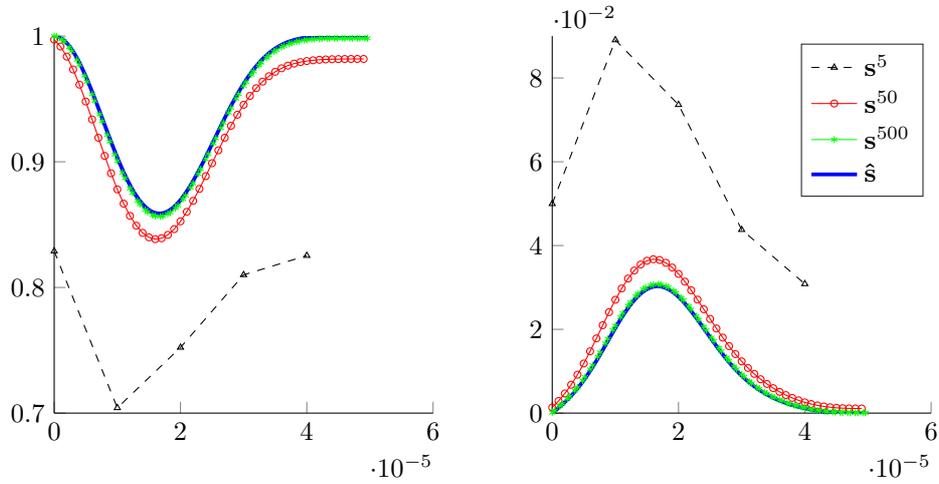}
	\caption{Quantity of interest $\mathbf{\hat{s}}$ and approximation $\mathbf{s}^{(M)}_n$ for Experiment 1. Real part on the left and imaginary part on the right.}
	\label{fig:exA_plotS}
\end{figure}

\subsubsection{Experiment 2: strong coupling}
MAS with four nuclear spins with homonuclear dipolar couplings leads to the Hamiltonian
\begin{align*}
	H(t) = \textrm{diag}\left[\{\alpha_k\}_{k=1}^{16}\right]+ B \cos(2\pi \nu t) + C \cos(4\pi \nu t),
\end{align*}
where $\alpha_k\in\mathbb{R}$ is a scalar, and $B,C\in\mathbb{R}^{16\times 16}$ are matrices with a sparsity structure as shown in Figure \ref{fig:strucB}.

\begin{figure}[!ht]
	\centering
	\scalebox{0.6}{\begin{tikzpicture} 
 \matrix (M) [matrix of nodes,left delimiter={[},right delimiter={]}] 
{$\times$&&&&&&&&&&&&&&&\\
&$\times$&$\times$&&$\times$&&&&$\times$&&&&&&&\\
&$\times$&$\times$&&$\times$&&&&$\times$&&&&&&&\\
&&&$\times$&&$\times$&$\times$&&&$\times$&$\times$&&&&&\\
&$\times$&$\times$&&$\times$&&&&$\times$&&&&&&&\\
&&&$\times$&&$\times$&$\times$&&&$\times$&&&$\times$&&&\\
&&&$\times$&&$\times$&$\times$&&&&$\times$&&$\times$&&&\\
&&&&&&&$\times$&&&&$\times$&&$\times$&$\times$&\\
&$\times$&$\times$&&$\times$&&&&$\times$&&&&&&&\\
&&&$\times$&&$\times$&&&&$\times$&$\times$&&$\times$&&&\\
&&&$\times$&&&$\times$&&&$\times$&$\times$&&$\times$&&&\\
&&&&&&&$\times$&&&&$\times$&&$\times$&$\times$&\\
&&&&&$\times$&$\times$&&&$\times$&$\times$&&$\times$&&&\\
&&&&&&&$\times$&&&&$\times$&&$\times$&$\times$&\\
&&&&&&&$\times$&&&&$\times$&&$\times$&$\times$&\\
&&&&&&&&&&&&&&&$\times$\\
};
 \end{tikzpicture}}
	\caption{Sparsity structure of $B$ and $C$ in Experiment 2, $\times$ denotes a nonzero element.}
	\label{fig:strucB}
\end{figure}
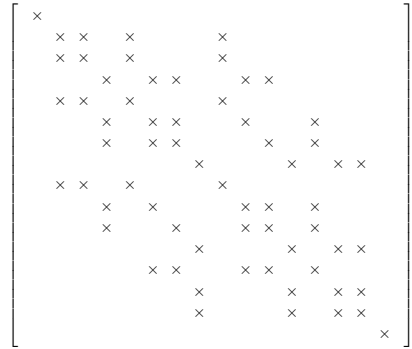
A typical experiment time is $1e-2$ seconds and $\nu = 1e+4$.
The simulated experiment time is $\tau_{\textrm{exp}} = 5e-6$, the size of the Krylov subspace is $k=4$ and $\mathbf{v}=\mathbf{w}=\begin{bmatrix}
	0 & 1 & 1 & 0 & 1 & 1& \dots & 0& 1& 1
\end{bmatrix}^\top$.
The corresponding error measures are shown in Table \ref{table:exB_KrylovErrors}, which show a similar behavior as for Experiment 1.
The measure $\textrm{err}_\textrm{sol}$ is computed by comparing $\mathbf{\hat{s}}$ to the solution obtained by \verb|ode45|.
\begin{table}[!ht]
	\centering
	\begin{tabular}{l|lll}
		$M$                         & 5       & 50      & 500           \\ \hline
		\hline
		$\textrm{err}_{\textrm{o}}$ &  3.596e-03  & 7.096e-05     & 7.028e-06    \\
		$\textrm{err}_{\textrm{V}}$ & 1.798e-16  & 4.573e-15  & 1.492e-13    \\
		$\textrm{err}_{\textrm{W}}$ & 3.783e-17   & 3.454e-17  & 3.508e-17 \\
		$\textrm{err}_\textrm{sol}$  &  3.877e-01  & 5.018e-02 &  4.281e-03
	\end{tabular}
	\caption{Error measures for Experiment 2.}
	\label{table:exB_KrylovErrors}
\end{table}

\subsubsection{Experiment 3: uncoupled spins under a pulse wave}
The Hamiltonian for four uncoupled spins under a pulse wave is
\begin{align*}
	H(t) = &\textrm{diag}\left[\{\alpha_k\}_{k=1}^{16}\right]+ B(0.5+\cos(4t) + \sin(10t) - 0.4 \sin(16t))\\
	&+ C(\sin(4t) + \cos(8t) + 2 \sin(12t)),
\end{align*}
with $\alpha_k\in\mathbb{R}$ {and $B\in\mathbb{R}^{16\times 16}$, $C\in\mathbb{C}^{16 \times 16}$} have {a structure as }shown in Figure~\ref{fig:strucC}.
\begin{figure}[!ht]
	\centering
	\scalebox{0.6}{\begin{tikzpicture} 
 \matrix (M) [matrix of nodes,left delimiter={[},right delimiter={]}] 
{&$\times$&$\times$&&$\times$&&&&$\times$&&&&&&&\\
$\times$&&&$\times$&&$\times$&&&&$\times$&&&&&&\\
$\times$&&&$\times$&&&$\times$&&&&$\times$&&&&&\\
&$\times$&$\times$&&&&&$\times$&&&&$\times$&&&&\\
$\times$&&&&&$\times$&$\times$&&&&&&$\times$&&&\\
&$\times$&&&$\times$&&&$\times$&&&&&&$\times$&&\\
&&$\times$&&$\times$&&&$\times$&&&&&&&$\times$&\\
&&&$\times$&&$\times$&$\times$&&&&&&&&&$\times$\\
$\times$&&&&&&&&&$\times$&$\times$&&$\times$&&&\\
&$\times$&&&&&&&$\times$&&&$\times$&&$\times$&&\\
&&$\times$&&&&&&$\times$&&&$\times$&&&$\times$&\\
&&&$\times$&&&&&&$\times$&$\times$&&&&&$\times$\\
&&&&$\times$&&&&$\times$&&&&&$\times$&$\times$&\\
&&&&&$\times$&&&&$\times$&&&$\times$&&&$\times$\\
&&&&&&$\times$&&&&$\times$&&$\times$&&&$\times$\\
&&&&&&&$\times$&&&&$\times$&&$\times$&$\times$&\\
};
 \end{tikzpicture}}
	\caption{Structure of $B$ and $C$ in Experiment 3, $\times$ denotes a nonzero element.}
	\label{fig:strucC}
\end{figure}
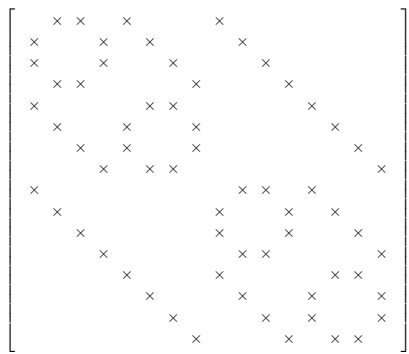
A practical experiment time ranges from $1e-6$ to $1e-3$ seconds, here $\tau_{\textrm{exp}} = 1e-3$ is used.
The starting vectors are $\mathbf{v}= \mathbf{w}= \begin{bmatrix}
	1 & \dots 1
\end{bmatrix}^\top$ and $n = 4$ iterations of the tensor Lanczos algorithm are run.
The Krylov error measures shown in Table \ref{table:exC_KrylovErrors} behave similarly to the measures observed for Experiment 1 and 2.
The measure $\textrm{err}_\textrm{sol}$ is obtained via \verb|ode45| and shows a convergence rate a bit slower than $\mathcal{O}(1/M)$.
The slower convergence rate {can, in part, be }explained by the fact that a comparison is made with the \verb|ode45| solution.
Additional errors are incurred when comparing $\mathbf{s}$ with $\mathbf{\hat{s}}$, because the former is only available in the points $\tau_i$ and the latter is available only in points which are determined by \verb|ode45|.
\begin{table}[!ht]
	\centering
	\begin{tabular}{l|lll}
		$M$                         & 5       & 50      & 500           \\ \hline
		 \hline
		$\textrm{err}_{\textrm{o}}$ & 2.994e-03   & 2.730e-05    & 2.708e-07   \\
		$\textrm{err}_{\textrm{V}}$ & 2.520e-16   & 4.122e-15 & 1.237e-13   \\
		$\textrm{err}_{\textrm{W}}$ & 1.203e-17   & 1.814e-17 & 1.008e-17 \\
		$\textrm{err}_\textrm{sol}$  &  3.476e-01 &  4.447e-02 & 5.145e-03
	\end{tabular}
	\caption{Error measures for Experiment 3.}
	\label{table:exC_KrylovErrors}
\end{table}

\subsection{{Tensor Train approximations}} \label{sec:TT_ranks}
As briefly mentioned in the Introduction, despite the fact that the block matrix and the tensor formulation of the problem \eqref{eq:ODE} are mathematically equivalent, the tensor formulation introduced and analyzed in this work, allows the exploitation of particular low parametric representations{. The aim} of this section is indeed to show that for all the examples previously presented, the resulting tensors can be accurately and conveniently approximated using a low parametric representation called Tensor Train (TT) format \cite{MR2837533,MR2566459}.

As a matter of fact, multilinear algebra, tensor analysis, and the theory of tensor approximations play increasingly important roles in nowadays computational mathematics and numerical analysis{, thereby} attracting tremendous interest in recent years \cite{cichocki2013tensor}.
In this panorama, Tensor-Train (TT) approximations {are a powerful technique} for dealing with \textit{the curse of dimensionality}, i.e., {the particularly unpleasant feature where the number of unknowns and the computational complexity grow exponentially when the dimension of the problem increases}.

Before presenting the computational results, we briefly survey the main features of the TT representation, addressing the interested reader to the surveys \cite{cichocki2013tensor,cichocki2016low}. 
We consider the Tensor Train (TT) format \cite{MR2837533}  for the  tensors of interest in this work. Specifically, a $4$-mode $\boldsymbol{\mathcal{A}} \in \mathbb{C}^{N_1 \times N_2 \times M \times M }$ tensor is expressed in TT format when
\begin{equation*}\label{eq:TT_format_4}
\boldsymbol{\mathcal{A}}_{i_1,i_2,i_3,i_4}=G_1(i_1)G_2(i_2)G_3(i_3)G_4(i_4)
\end{equation*}
where $G_k(i_k)$ is a matrix of dimension $r_{k-1}\times r_k$ and $r_0=r_4=1$. 
The numbers $r_k$ are called \emph{TT-ranks}, and $G_k(i_k)$ are the \emph{cores} of the TT-decomposition. If $r_k \leq r$, $n_k \leq n$, then storing the TT-representation requires memorizing $ \leq 4nr^2$ numbers. If $r$ is small, then the memory requirement is much smaller than storing the full tensor, i.e, storing $n^4$ numbers. 

It is important to note that the TT representation allows to approximate various tensor operations efficiently, see, e.g., \cite[Sec. 4]{MR2837533}.
In this paper, we do not propose a low parametric TT version of Algorithm \ref{alg:Lanczos}. To be efficient, such a TT version would need a TT representation of the tensor products used in  Algorithm \ref{alg:Lanczos}. This paper aims to show that Algorithm \ref{alg:Lanczos} works; further enhancements are postponed to future investigations.

In Tables \ref{table:expA} - \ref{table:expC} we present the TT ranks for all the tensors considered in Section \ref{sec:exp_NMR}. In particular, the tables present the details of the TT approximations obtained using the \texttt{TT-toolbox}  \cite{MR2837533} when the required accuracy for the approximation is set to \texttt{$1e-5$} and \texttt{$1e-10$}.
As becomes evident from the presented results, all the considered tensors are amenable of a low parametric representation provided by the TT format and, indeed, for all the presented results the Compression Factor (C.F.), which is defined as $(\sum_{k=1}^{4} r_{k-1} \times n_k \times r_{k})/nnz(\boldsymbol{\mathcal{A}})$, with $nnz(\boldsymbol{\mathcal{A}})$ the number of nonzero elements of $\boldsymbol{\mathcal{A}}$, lies in the interval $(1e-3, 0.5 )$. It is important to observe that when increasing the accuracy from \texttt{$1e-5$} to \texttt{$1e-10$} the C.F. does not significantly change, suggesting the interesting fact  that for the considered tensors, the TT format is closer to an \textit{exact representation} rather than an \textit{approximation}.
Finally, it is important to note that the ranks of the TT approximations are robust across the choices of the parameter $\omega$ (cfr. the TT ranks in  Table \ref{table:expA}  and in  Table \ref{table:expB})
and to note also that, for some of the considered problems, the number of parameters needed for the \textit{approximation} can be two orders of magnitude smaller than $nnz(\boldsymbol{\mathcal{A}})$; {see Tables \ref{table:expB} and \ref{table:expC}}.
\begin{table} [ht!]
	\centering
	\caption{Experiment 1}
\begin{tabular}{lllll}\label{table:expA}
	$M$& $nnz(\boldsymbol{\mathcal{A}})$ & TT Ranks & $\sum_{k=1}^{4} r_{k-1} \times n_k \times r_{k}  $ & C.F. \\ 
	\hline
	\hline	 
	& & $\nu=1e+4,\; tol=1e-5$ & & \\
	\hline
	\hline	
	500 & 2004000 & 1   16    2  498    1 & 747768 & 0.37314 \\ 
	1000 & 8008000 & 1   16    2  900    1 & 2700768 & 0.33726 \\ 
	1500 & 18012000 & 1    16     2  1476     1 & 6642768 & 0.3688 \\ 
		\hline
	\hline 
	& & $\nu=1e+4,\; tol=1e-10$ & & \\
	\hline
    \hline
	500 & 2004000 & 1   16    2  500    1 & 750768 & 0.37463 \\ 
	1000 & 8008000 & 1   16    2  998    1 & 2994768 & 0.37397 \\ 
	1500 & 18012000 & 1    16     2  1500     1 & 6750768 & 0.37479 \\ 
	\hline 
	\hline 
	& & $\nu=1e+1,\; tol=1e-10$ & & \\
	\hline
	\hline
	500 & 2004000 & 1   16    2  500    1 & 750768 & 0.37463 \\ 
	1000 & 8008000 & 1    16     2  1000     1 & 3000768 & 0.37472 \\ 
	1500 & 18012000 & 1    16     2  1500     1 & 6750768 & 0.37479 \\
	\hline
	\hline		
\end{tabular}
\end{table}

\begin{table} [ht!]
	\centering
	\caption{Experiment 2}
	\begin{tabular}{lllll}\label{table:expB}
		$M$ & $nnz(\boldsymbol{\mathcal{A}})$ & TT Ranks & $\sum_{k=1}^{4} r_{k-1} \times n_k \times r_{k}  $ & C.F. \\ 
		\hline
	    \hline		 
		& & $\nu=1e+4,\; tol=1e-5$ & & \\
		\hline
		\hline	
		500 & 8016000 & 1   15    1  498    1 & 498480 & 0.062186 \\ 
		1000 & 29627200 & 1   15    1  900    1 & 1800480 & 0.060771 \\ 
		1500 & 72048000 & 1    15     1  1476     1 & 4428480 & 0.061466 \\ 
		
		\hline
		\hline	 
		& & $\nu=1e+4,\; tol=1e-10$ & & \\
		\hline
		\hline	
		500 & 8016000 & 1   16    2  500    1 & 750768 & 0.093659 \\ 
		1000 & 29627200 & 1   16    2  997    1 & 2991768 & 0.10098 \\ 
		1500 & 72048000 & 1    16     2  1500     1 & 6750768 & 0.093698 \\
		\hline
		\hline	 
		& & $\nu=1e+1,\; tol=1e-10$ & & \\
		\hline
		\hline
		500 & 8016000 & 1   16    2  500    1 & 750768 & 0.093659 \\ 
		1000 & 32032000 & 1    16     2  1000     1 & 3000768 & 0.09368 \\ 
		1500 & 72048000 & 1    16     2  1500     1 & 6750768 & 0.093698 \\ 
		\hline
		\hline	
	\end{tabular}
\end{table}

\begin{table} [ht!]
	\centering
	\caption{Experiment 3}
	\begin{tabular}{lllll}\label{table:expC}
		$M$& $nnz(\boldsymbol{\mathcal{A}})$ & TT Ranks & $\sum_{k=1}^{4} r_{k-1} \times n_k \times r_{k}  $ & C.F. \\ 
		\hline
		\hline	 
		&  & $\nu=1e+4,\; tol=1e-5$ & & \\
		\hline
		\hline	
		500 & 10020000 & 1   12    2  499    1 & 749076 & 0.074758 \\ 
		1000 & 40040000 & 1   12    2  996    1 & 2988576 & 0.07464 \\ 
		1500 & 90060000 & 1    12     2  1493     1 & 6719076 & 0.074607 \\

		\hline
		\hline	 
		& & $\nu=1e+4,\; tol=1e-10$ & & \\
		\hline
		\hline	
		500 & 10020000 & 1   16    2  500    1 & 750768 & 0.074927 \\ 
		1000 & 40040000 & 1    16     2  1000     1 & 3000768 & 0.074944 \\ 
		1500 & 90060000 & 1    16     2  1500     1 & 6750768 & 0.074959 \\ 
		\hline
		\hline	 
	\end{tabular}
\end{table}

\section{Conclusions}\label{sec:conclusions}
In this work we introduced a non-Hermitian Lanczos algorithm for tensors
and we provided the corresponding theoretical analysis. In particular, after introducing all the necessary theoretical background, we are able to interpret such Lanczos-type process in terms of tensor polynomials and to prove the related matching moment property. 
A series of numerical experiments performed on real world problems confirm the effectiveness of our approach. Using a linearly converging approximation for the inputs, the algorithm produces a linearly converging approximation of the bilinear form  $\mathbf{w}^H \boldsymbol{\mathsf{U}}(t) \mathbf{v}$, where  $\boldsymbol{\mathsf{U}}(t)$ is the solution of the ODE \eqref{eq:ODE}. More accurate approximation schemes for the inputs are currently being developed by some of the paper's authors, possibly leading to a faster convergence.
Moreover, in all the considered examples, the related tensors show a low parametric structure in terms of Tensor Train representation. This important feature paves the path for future efficiency improvements of our proposal where {this} representation is fully exploited in the Lanczos-type procedure.

\section*{Acknowledgements} 
{The authors want to thank Enik\"o Balig\'acs and Christian Bonhomme (Laboratoire de chimie de la matière condensée de Paris, Sorbonne University) for providing the data from real world applications used in Section \ref{sec:ex}. This work was supported by Charles University Research programs No. PRIMUS/21/SCI/009 and UNCE/SCI/023, and by the Magica project ANR-20-CE29-0007 funded by the French National Research Agency. The authors S.P. and M.R.-Z. are members of the GNCS-INdAM group.}

\newpage 

\appendix

\section{Appendix} \label{appendix}

\subsection{The $\ast$-product associativeness}\label{app:assoc}
In the following, we prove the three statments of Lemma \ref{lemma:properties_defintions1}.
	First, we have
	\begin{equation*}
	\begin{split}
	& ((\boldsymbol{\mathcal{A}}*\boldsymbol{\mathcal{A}})*{A})_{i_1,:,:}=\sum_{k=1}^{N_1} (\boldsymbol{\mathcal{A}}*\boldsymbol{\mathcal{A}})_{i_1,k,:,:}A_{k,:,:}=\sum_{k=1}^{N_1}(\sum_{j=1}^{N_1} \boldsymbol{\mathcal{A}}_{i_1,j,:,:}\boldsymbol{\mathcal{A}}_{j,k,:,:})A_{k,:,:}= \\
	& \sum_{j=1}^{N_1} \boldsymbol{\mathcal{A}}_{i_1,j,:,:}(\sum_{k=1}^{N_1}\boldsymbol{\mathcal{A}}_{j,k,:,:}A_{k,:,:})=\sum_{j=1}^{N_1} \boldsymbol{\mathcal{A}}_{i_1,j,:,:}(\boldsymbol{\mathcal{A}}*A)_{j,:,:}=(\boldsymbol{\mathcal{A}}*(\boldsymbol{\mathcal{A}}*A))_{i_1,:,:}.
	\end{split} 
	\end{equation*}
	The second statement is proved by direct inspection as follows:
	\begin{equation*}
	\begin{split}
	& ((B^D * \boldsymbol{\mathcal{A}})^D *A)_{:,:}= \sum_{k=1}^{N_2}(B^D * \boldsymbol{\mathcal{A}})^D_{k,:,:}A_{k,:,:} = 
	\sum_{k=1}^{N_2}(\sum_{j=1}^{N_1} B^D_{j,:,:} \boldsymbol{\mathcal{A}}_{j,k,:,:})A_{k,:,:}=\\
	&	\sum_{j=1}^{N_1} B^D_{j,:,:}(\sum_{k=1}^{N_2} \boldsymbol{\mathcal{A}}_{j,k,:,:}A_{k,:,:})=\sum_{j=1}^{N_1} B^D_{j,:,:}(\mathcal{A}*A)_{j,:,:} = B^D * (\boldsymbol{\mathcal{A}} *A).
	\end{split}
	\end{equation*}
	Finally, we have the following equality
	\begin{equation*}
	\begin{split}
	&((\boldsymbol{\mathcal{C}}*\boldsymbol{\mathcal{A}})*\boldsymbol{\mathcal{B}})_{i_1,i_2,:,:}=\sum_{k=1}^{N_2}(\boldsymbol{\mathcal{C}}*\boldsymbol{\mathcal{A}})_{i_1,k,:,:}\boldsymbol{\mathcal{B}}_{k,i_2,:,:} = \sum_{k=1}^{N_2}(\sum_{j=1}^{N_1}\boldsymbol{\mathcal{C}}_{i_1,j,:,:}\boldsymbol{\mathcal{A}}_{j,k,:,:})\boldsymbol{\mathcal{B}}_{k,i_2,:,:}=\\
	&\sum_{j=1}^{N_1}\boldsymbol{\mathcal{C}}_{i_1,j,:,:}(\sum_{k=1}^{N_2}\boldsymbol{\mathcal{A}}_{j,k,:,:}\boldsymbol{\mathcal{B}}_{k,i_2,:,:})=\sum_{j=1}^{N_1}\boldsymbol{\mathcal{C}}_{i_1,j,:,:}(\boldsymbol{\mathcal{A}}*\boldsymbol{\mathcal{B}})_{j,i_2,:,:}=(\boldsymbol{\mathcal{C}}*(\boldsymbol{\mathcal{A}}*\boldsymbol{\mathcal{B}}))_{i_1,i_2,:,:}.
	\end{split}
	\end{equation*}
	This concludes the Lemma proof.

\subsection{Matching moment property}\label{appendix:mmp}
We prove Theorem \ref{thm:mmp} by giving a polynomial interpretation of the Lanczos process for tensors. The proof follows the same principles as the proof given in \cite{GisPoz22} for a different but analogous case.
For simplicity, we consider the case in which the matrices $\boldsymbol{\gamma}_k$ are set equal to the matrix identity. The proof can be easily extended to the general case.
Let us define the following set of polynomials 
\begin{equation*}
    \mathcal{P}_* := \left\{ p(\lambda) = \sum_{{k}=0}^{{\ell}} \lambda^{{k_*}}\times \boldsymbol{\eta}_{{k}} \right\},
\end{equation*}
with $\boldsymbol{\eta}_{{k}} \in \mathbb{C}^{M \times M}$.
We say that the map $[\cdot, \cdot]: \mathcal{P}_* \times \mathcal{P}_* \rightarrow \mathbb{C}^{M\times M}$ is a \emph{sesquilinear block form} if and only if, given $p,q,r,s \in \mathcal{P}_\ast$ and $\boldsymbol{\alpha}, \boldsymbol{\beta} \in \mathbb{C}^{M \times M}$, it satisfies 
\begin{align*}
    [q \times \boldsymbol{\alpha}, p \times \boldsymbol{\beta}] &= \boldsymbol{\alpha}^H \times[q,p] \times \boldsymbol{\beta}, \\
    [q + s, p + r] &= [q, p] + [s, p] + [q, r] + [s, r].
\end{align*}
In addition, from now on we assume that for every sesquilinear block form $[\cdot, \cdot]$ it holds that
\begin{equation}\label{eq:bfor:prop}
    [\lambda*q,p] = [q, \lambda*p].
\end{equation}
Then, $[\cdot,\cdot]$ is determined by its \emph{moments} defined as
\begin{equation*}
    \boldsymbol{\mu}_{{k}}:= [\lambda^{{{k_*}}},1] = [1, \lambda^{{{k}_*}}], \quad {k}=0,1,\dots \,. 
\end{equation*}

We say that the sequences $p_0, p_1, \dots$ and $q_0, q_1, \dots$ from $\mathcal{P}_*$ are sequences of \emph{biorthonormal polynomimals} with respect to $[\cdot,\cdot]$ if and only if 
\begin{equation}\label{eq:p:orth:cond}
    [q_i, p_j] = \delta_{ij} I_M,
\end{equation}
with $\delta_{ij}$ the Kronecker delta (hereafter the subindex ${k}$ in $p_{{k}}$ and $q_{{k}}$ will stand for the degree of the polynomial).
From now on, we also assume $m_0 = I_M$, getting $p_0 = q_0 = I_M$.
Let $q_1$ be the polynomial from $\mathcal{P}_*$
\begin{equation*}
   q_1(\lambda) = \lambda * q_0(\lambda) - q_0(\lambda)\times \boldsymbol{\alpha}_0^H,
\end{equation*}
the orthogonality conditions \eqref{eq:p:orth:cond} imply
\begin{equation*}
    \boldsymbol{\alpha}_0 = [\lambda*q_0, p_0].
\end{equation*}
{Analogously}, we get
\begin{equation*}
   p_1(\lambda) \times \boldsymbol{\beta}_1 = \lambda * p_0(\lambda) - p_0(\lambda) \times \boldsymbol{\alpha}_0,
\end{equation*}
with
\begin{equation*}
    \boldsymbol{\alpha}_0 = [q_0, \lambda*p_0], \quad  \boldsymbol{\beta}_1 = [q_1, \lambda*p_0].
\end{equation*}
Repeating such a biorthogonalization process, we obtain the three-term recurrences for $k=0,1,\dots$ 
\begin{subequations}\label{eq:p:rec}
\begin{align}
     q_{k+1}(\lambda) &= \lambda * q_{k}(\lambda) - q_{k}(\lambda)\times \boldsymbol{\alpha}_{k}^H - q_{k-1}(\lambda)\times \boldsymbol{\beta}_{k}^H \label{eq:p:rec:1} \\
     p_{k+1}(\lambda)\times\boldsymbol{\beta}_{k+1} &= \lambda * p_{k}(\lambda) - p_{k}(\lambda)\times\boldsymbol{\alpha}_{k} - p_{k-1}(\lambda),
     \label{eq:p:rec:2}
\end{align}
\end{subequations}
with $p_{-1} = q_{-1} = 0$ and
\begin{equation}\label{eq:p:coeff}
    \boldsymbol{\alpha}_{k} = [q_{k}, \lambda * p_{k}],
    \quad
    \boldsymbol{\beta}_{k+1} = [q_{k+1}, \lambda * p_{k}].
\end{equation}
We remark that the recurrences are obtained using property \eqref{eq:bfor:prop}.
The previous derivation also constructively proves that the biorthonormal polynomials $p_0, \dots, p_n$ and $q_0, \dots, q_n$ exist if $\boldsymbol{\beta}_1, \dots, \boldsymbol{\beta}_n$ are invertible matrices.

Let $\boldsymbol{\mathcal{A}}, V, W$ be as in Theorem \ref{thm:mmp} and let us defined the sesquilinear block form 
\begin{equation*}
    [q, p]_{\boldsymbol{\mathcal{A}}} = W^D * q^D(\boldsymbol{\mathcal{A}}) * p(\boldsymbol{\mathcal{A}}) * V.
\end{equation*}
Assume that there exist polynomials $p_0, \dots, p_n$ and $q_0,\dots,q_n$ from $\mathcal{P}_*$ which are biorthonormal with respect to $[\cdot, \cdot]_{\boldsymbol{\mathcal{A}}}$.
Defining the vectors
{\begin{equation*}
    V_k = p_{k-1}(\boldsymbol{\mathcal{A}})*V, \qquad W^D_k = W^D * q_{k-1}^D(\boldsymbol{\mathcal{A}}), 
\end{equation*}}
and using the recurrences \eqref{eq:p:rec} we get the recurrences \eqref{eq:3r} of the non-Hermitian Lanczos for tensors.
Moreover, the coefficients in \eqref{eq:p:coeff} are the coefficients in \eqref{eq:3r:coeff}.

Let $\boldsymbol{\mathcal{T}}_n$ be as in Theorem \ref{thm:mmp}.
We can define the sesquilinear block form $[q, p]_n: \mathcal{P}_* \times \mathcal{P}_* \rightarrow \mathbb{C}^{M \times M}$ as
\begin{equation*}
    [q, p]_n = E_1^D * q^D(\boldsymbol{\mathcal{T}}_n)* p(\boldsymbol{\mathcal{T}}_n)* E_1.
\end{equation*}
Note that here the vector $\mathbf{e}_1$ in the definition of $E_1 = \mathsf{e}_1 \otimes I_M$ has length $n \leq N$.
The following Lemmas will show that
\begin{equation*}
   \boldsymbol{\mu}_{{k}} = [\lambda^{*{k}},1]_{\boldsymbol{\mathcal{A}}} = [\lambda^{*{k}},1]_n, \quad {k}=0,\dots,2n-1,
\end{equation*}
concluding the proof of Theorem \ref{thm:mmp}.

\begin{lemma}\label{lemma:op:basis}
   Let $p_0,\dots,p_n \in \mathcal{P}_*$ and $q_0,\dots,q_n \in \mathcal{P}_*$ be biorthonormal polynomials with respect to $[\cdot,\cdot]_{\boldsymbol{\mathcal{A}}}$.
   Assume that $\boldsymbol{\beta}_1,\dots, \boldsymbol{\beta}_n$ in \eqref{eq:p:rec} are invertible matrices.
   Then the polynomials are also biorthonormal with respect to $[\cdot,\cdot]_n$ ad defined above. 
\end{lemma}
\begin{proof}
Let us define the tensors $E_i := \mathbf{e}_{i} \otimes I_M$ for $i=1,\dots, n$, with $I_n = [\mathsf{e}_1, \dots, \mathsf{e}_n]$.
 We will first prove by induction that for $i=0,\dots,n-1$ 
   \begin{equation}\label{eq:lem:op:basis:Ei}
      E_{i+1} = p_i(\boldsymbol{\mathcal{T}}_n) * E_1,
      \quad
      E_{i+1}^D = E_1^D * q_i^D(\boldsymbol{\mathcal{T}}_n).
  \end{equation}
  For $i=0$, the Eq.~\eqref{eq:lem:op:basis:Ei} are trivial. 
  Assume now that Eq.~\eqref{eq:lem:op:basis:Ei} hold for $i=1,\dots,k$,
  by \eqref{eq:p:rec} we get
  \begin{align*}
     E_1^D * q_{k+1}^D(\boldsymbol{\mathcal{T}}_n) &= E_{k+1}^D * \boldsymbol{\mathcal{T}}_n - \boldsymbol{\alpha}_{k} \times E_{k+1}^D - \boldsymbol{\beta}_{k} \times E_{k}^D,  \\
     p_{k+1}(\boldsymbol{\mathcal{T}}_n)\times\boldsymbol{\beta}_{k+1} &= \boldsymbol{\mathcal{T}}_n * E_{k+1} - E_{k+1}\times\boldsymbol{\alpha}_{k} - E_k.
\end{align*}
Since $\boldsymbol{\beta}_{k+1}$ is invertible, direct computations prove that \eqref{eq:lem:op:basis:Ei} holds $i=k+1$.

As a consequence we have
\begin{equation*}
    [{q}_i, {p}_j]_n = E_{i+1}^D \ast E_{j+1} = \delta_{ij} I_M,
\end{equation*}
which concludes the proof.
\end{proof}

\begin{lemma}
   Let $p_0,\dots,p_{n-1} \in \mathcal{P}_*$ and $q_0,\dots,q_{n-1} \in \mathcal{P}_*$ be biorthonormal polynomials with respect to a sesquilinear block form $[\cdot,\cdot]_{\boldsymbol{\mathcal{A}}}$
   and to a sesquilinear block form $[\cdot,\cdot]_{\boldsymbol{\mathcal{B}}}$.
   If $[1,1]_{\boldsymbol{\mathcal{A}}} = [1,1]_{\boldsymbol{\mathcal{B}}} = I_M$,
   then $[\lambda^{{k_*}},1]_{\boldsymbol{\mathcal{A}}} = [\lambda^{{k}_*},1]_{\boldsymbol{\mathcal{B}}}$ for ${k}=0,\dots,2n-1$.
\end{lemma}
\begin{proof}
  The proof is by induction. 
  Let $\boldsymbol{\mu}_{{k}} = [\lambda^{{{k}_*}},1]_{\boldsymbol{\mathcal{A}}}$ and $\widehat{\boldsymbol{\mu}}_j = [\lambda^{{{k}_*}},1]_{\boldsymbol{\mathcal{B}}} $ for ${k}=0,1,\dots, 2n-1$.
  The coefficient formula \eqref{eq:p:coeff} gives
  \begin{equation*}
      [q_{0}, \lambda*p_{0}]_{\boldsymbol{\mathcal{A}}} = \boldsymbol{\alpha}_0 = [q_{0}, \lambda*p_{0}]_{\boldsymbol{\mathcal{B}}}.
  \end{equation*}
  Hence $\boldsymbol{\mu}_1 = \boldsymbol{\alpha}_0 = \widehat{\boldsymbol{\mu}}_1$.
  Considering the induction assumptions $\boldsymbol{\mu}_{{k}} = \widehat{\boldsymbol{\mu}}_{{k}}$ for ${k}=0,\dots,2{j}-3$,
  we prove that $\boldsymbol{\mu}_{2{j}-2} = \widehat{\boldsymbol{\mu}}_{2{j}-2}$
  and $\boldsymbol{\mu}_{2{j}-1} = \widehat{\boldsymbol{\mu}}_{2{j}-1}$, for ${j} = 2,\dots,n$.
  By the formula in \eqref{eq:p:coeff} we get
  \begin{equation*}
      [q_{{j}-1}, \lambda*p_{{j}-2}]_{\boldsymbol{\mathcal{A}}} = \boldsymbol{\beta}_{{j}-1} = [q_{{j}-1}, \lambda*p_{{j}-2}]_{\boldsymbol{\mathcal{B}}},
  \end{equation*}
  which we can rewrite as
  \begin{equation*}
      \sum_{i=0}^{{j}-1}\sum_{{k}=0}^{{j}-2} \boldsymbol{\eta}_i^H \times \boldsymbol{\mu}_{i+{k}+1} \times \widehat{\boldsymbol{\eta}}_{{k}} = \sum_{i=0}^{{j}-1}\sum_{{k}=0}^{{j}-2} \boldsymbol{\eta}_i^H \times \widehat{\boldsymbol{\mu}}_{i+{k}+1} \times \widehat{\boldsymbol{\eta}}_{{k}},
  \end{equation*}
  where $\boldsymbol{\eta}_i, \widehat{\boldsymbol{\eta}}_{{k}} \in \mathbb{C}^{M \times M}$ are the coefficients respectively of $q_{{j}-1}$ and $p_{{j}-2}$.
  By the induction assumption we obtain
    \begin{equation*}
       \boldsymbol{\eta}_{{j}-1}^H \times \boldsymbol{\mu}_{2{j}-2} \times \widehat{\boldsymbol{\eta}}_{{j}-2} = \boldsymbol{\eta}_{{j}-1}^H \times \widehat{\boldsymbol{\mu}}_{2{j}-2} \times \widehat{\boldsymbol{\eta}}_{{j}-2}.
  \end{equation*}
  The leading coefficients of the polynomials $q_{2{j}-2}$ and $p_{2{j}-2}$ are respectively $\boldsymbol{\eta}_{{j}-1}= 1$ and $\widehat{\boldsymbol{\eta}}_{{j}-2}=(\boldsymbol{\beta}_{{j}-2} \times \cdots \times \boldsymbol{\beta}_{1})^{-1}$. Hence $\boldsymbol{\mu}_{2{j}-2} = \widehat{\boldsymbol{\mu}}_{2{j}-2}$.
  We conclude the proof repeating the same argument with the coefficient $\boldsymbol{\alpha}_{{j}-1}$ \eqref{eq:p:coeff}.
\end{proof}

\section*{Data Availability Statement}
Data sharing not applicable to this article as no datasets were generated or analyzed during the current study.

\section*{Conflict of Interest}
The authors declare that they have no conflict of interest

\bibliographystyle{abbrv}
\bibliography{bibliography}

\end{document}